\documentclass{amsart}
\usepackage{graphicx, color}
\usepackage{amscd}
\usepackage{amsmath,empheq}
\usepackage{amsfonts}
\usepackage{amssymb}
\usepackage{mathrsfs}
\usepackage[all]{xy}
\newtheorem{theorem}{Theorem}[section]

\newtheorem{lemma}[theorem]{Lemma}
\newtheorem{prop}[theorem]{Proposition}
\newtheorem{remark}{Remark}[section]
\newtheorem{corollary}[theorem]{Corollary}
\newtheorem{claim}{Claim}[section]

\newtheorem{example}{Example}[section]
\newenvironment{proof-sketch}{\noindent{\bf Sketch of Proof}\hspace*{1em}}{\qed\bigskip}

\everymath{\displaystyle}
\newcommand{\RR}{\mathbb R}
\newcommand{\NN}{\mathbb N}

\newcommand{\ZZ}{\mathbb Z}

\renewcommand{\leq}{\leqslant}

\renewcommand{\geq}{\geqslant}
\date{\today}
\begin{document}
\title[Robin problems with indefinite potential and general reaction]{Nonlinear, nonhomogeneous Robin problems with\\ indefinite potential and general reaction}
\author[N.S. Papageorgiou]{N.S. Papageorgiou}
\address[N.S. Papageorgiou]{ Department of Mathematics, National Technical University,
				Zografou Campus, Athens 15780, Greece \& Institute of Mathematics, Physics and Mechanics, Jadranska 19, 1000 Ljubljana, Slovenia}
\email{\tt npapg@math.ntua.gr}
\author[V.D. R\u{a}dulescu]{V.D. R\u{a}dulescu}
\address[V.D. R\u{a}dulescu]{Institute of Mathematics, Physics and Mechanics, Jadranska 19, 1000 Ljubljana, Slovenia \& Faculty of Applied Mathematics, AGH University of Science and Technology, al. Mickiewicza 30, 30-059 Krak\'ow, Poland \& Institute of Mathematics ``Simion Stoilow" of the Romanian Academy, P.O. Box 1-764, 014700 Bucharest, Romania}
\email{\tt vicentiu.radulescu@imfm.si}
\author[D.D. Repov\v{s}]{D.D. Repov\v{s}}
\address[D.D. Repov\v{s}]{Faculty of Education and Faculty of Mathematics and Physics, University of Ljubljana \& Institute of Mathematics, Physics and Mechanics, Jadranska 19, 1000 Ljubljana, Slovenia}
\email{\tt dusan.repovs@guest.arnes.si}
\begin{abstract}
We consider a nonlinear elliptic equation driven by a nonhomogeneous differential operator plus an indefinite potential. On the reaction term we impose conditions only near zero. Using variational methods, together with truncation and perturbation techniques and critical groups, we produce three nontrivial solutions with sign information. In the semilinear case we improve this result by obtaining a second nodal solution for a total of four nontrivial solutions. Finally, under a symmetry condition on the reaction term, we generate a whole sequence of distinct nodal solutions.
\end{abstract}
\maketitle

\section{Introduction}

Let $\Omega\subseteq\RR^N$ be a bounded domain with a $C^2$-boundary $\partial\Omega$. In this paper we study the following nonlinear nonhomogeneous Robin problem
\begin{equation}\label{eq1}
	\left\{\begin{array}{ll}
		- {\rm div}\, a(Du(z)) + \xi(z)\arrowvert u(z)\arrowvert^{p-2}u(z)= f(z,u(z))\quad \mbox{in}\ \Omega,\\
		\frac{\partial u}{\partial n_{a}} + \beta(z)|u|^{p-2}u=0\ \mbox{on}\quad \partial\Omega\,.
		\end{array}\right\}
\end{equation}
In this problem, the map $a:\RR^{\NN}\rightarrow\RR^{\NN}$ involved in the differential operator is a continuous, strictly monotone (thus maximal monotone operator, too) map which satisfies certain other regularity and growth conditions listed in hypotheses $H(a)$ below. These conditions are general enough to generate a broad framework that incorporates many differential operators of interest, such as the $p$-Laplacian and the $(p,q)$-Laplacian (that is, the sum of a $p$-Laplacian and a $q$-Laplacian, with $1<q<p<\infty$). Note that in general, the differential operator $u\mapsto {\rm div}\, a(Du)$ is not homogeneous. The potential function $\xi(\cdot)\in L^{\infty}(\Omega)$ and in general, $\xi(\cdot)$ is  nodal (that is, sign changing). So, the left-hand side of problem (\ref{eq1}) needs not be coercive. The reaction term $f(z,x)$ is a Carath\'eodory function (that is, for all $x\in\RR$, the mapping $ z\mapsto f(z,x)$ is measurable, while for almost all $z\in\Omega$, the mapping $ x\mapsto f(z,x)$ is continuous). The special feature of our paper  is that no global growth condition is imposed on $f(z,\cdot)$. The only conditions imposed on $f(z,\cdot)$ concern its behavior near zero and that $f(z,\cdot)$ must be locally $L^{\infty}$-bounded. In the boundary condition, $\frac{\partial u}{\partial n_{a}}$ denotes the generalized normal derivative corresponding to the map $a(\cdot)$. It is defined by extension of the map
$$C^{1}(\overline{\Omega})\ni u \mapsto \frac{\partial u}{\partial n_{a}} = (a(Du), n)_{\RR_\NN},$$
with $n(\cdot)$ being the outward unit normal on $\partial\Omega$. This kind of conormal derivative is dictated by the nonlinear Green identity (see Gasinski and Papageorgiou \cite[p. 210]{6}) and was also used  by Lieberman \cite{11} in his nonlinear regularity theory. The boundary coefficient is $\beta\in C^{0,\alpha}(\partial\Omega)$, with $0<\alpha<1$ and $\beta(z)\geq0$ for all $z\in\partial\Omega$.

The aim of the present paper is to prove a multiplicity theorem for such equations, providing sign information for all solutions produced. Wang \cite{28} was the first to study elliptic problems with a general reaction term of arbitrary growth. The equation studied by Wang \cite{28} was a nonlinear problem driven by Dirichlet $p$-Laplacian with zero potential. Using cut-off techniques and imposing a symmetry condition on $f(z,\cdot)$ (that is, assuming that $f(z,\cdot)$ is odd), Wang \cite{28} produced an infinity of nontrivial solutions. More recently, Li and Wang \cite{12}, using similar tools, improved this result by producing an infinity of nodal solutions for semilinear Schr\"{o}dinger equations. Their result was extended by Papageorgiou and R\u{a}dulescu \cite{18} who considered nonlinear, nonhomogeneous Robin problems with zero potential (that is, $\xi\equiv0$). Assuming that the reaction term $f(z,\cdot)$ has zeros of constant sign and that it is odd, they produced an infinity of  smooth nodal solutions. We also mention the recent work of Papageorgiou and Winkert \cite{22}, who considered a reaction term of general growth and with zeros. Under stronger conditions on the map $a(\cdot)$ and with zero potential, they produced constant sign and nodal solutions.
 We refer to  Pucci {\it at al.} \cite{patri1, patri2} for eigenvalue problems associated to $p$-Laplacian type operators.
Related results in the framework of problems with unbalanced growth are due to Fiscella and Pucci \cite{patri3}, and Papageorgiou, R\u{a}dulescu and Repov\v{s} \cite{prrana18}.
Finally, we also point out the papers of He, Yao and Sun \cite{8} on nonlinear, nonhomogeneous Neumann problems with nonnegative potential (that is, $\xi\geq0$, $\xi\not\equiv0$), Iturriaga, Massa, Sanchez and Ubilla \cite{9} on parametric equations driven by Dirichlet $p$-Laplacian with zero potential and a reaction with zeros, and  Tan and Fang \cite{26} on nonlinear, nonhomogeneous Dirichlet problems using the formalism of Orlicz spaces.

\section{Mathematical Background}

Let $X$ be a Banach space and $X^{*}$ its topological dual. By $\left\langle \cdot,\cdot\right\rangle$ we denote the duality brackets for the pair $(X^{*},X)$. Given $\varphi\in C^{1}(X,\RR)$, we say that $\varphi$ satisfies the ``Cerami condition" (the C-condition for short), if the following property holds:
\begin{center}
``Every sequence $\{u_n\}_{n\geq1}\subseteq X$ such that $\{\varphi(u_n)\}_{n\geq1}\subseteq\RR$ is bounded and
$$(1+||u_n||)\varphi'(u_n)\rightarrow0\ \mbox{in}\ X^{*}\ \mbox{as}\ n\rightarrow\infty,$$
admits a strongly convergent subsequence".
\end{center}

This compactness-type condition on the functional $\varphi$, leads to a deformation theorem from which one can derive the minimax theory for the critical values of $\varphi$. A fundamental result in this theory is the so-called ``mountain pass theorem".

\begin{theorem}\label{th1}
	Let $X$ be a Banach space and assume that $\varphi\in C^{1}(X,\RR)$ satisfies the C-condition, $u_0, u_1\in X$, $||u_1-u_0||>\rho>0$,
	$$\max\{\varphi(u_0),\varphi(u_1)\}\ <\ \inf\left\{\varphi(u):||u-u_0||=\rho\right\} = m_\rho$$
	and
	$$c = \inf_{\gamma\in\Gamma}\max_{0\leq t\leq1}\varphi(\gamma(t))\ \mbox{with}\ \Gamma=\{\gamma\in C([0,1],X):\gamma(0)=u_0,\gamma(1)=u_1\}.$$
	Then $c\geq m_\rho$ and $c$ is a critical value of $\varphi$.
\end{theorem}

Let $k\in C^{1}(0,\infty)$ with $k(t)>0$ for all $t>0$. We assume that
\begin{equation}\label{eq2}
	0<\hat{c}\leq\frac{k'(t)t}{k(t)}\leq c_{0}\ \mbox{and}\ c_{1}t^{p-1}\leq k(t)\leq c_{2}(1+t^{p-1})\ \mbox{for all}\ t>0\ \mbox{, with}\ c_1,c_2>0.
\end{equation}

We introduce the following conditions on the map $\alpha(\cdot)$ (see also Papageorgiou and R\u{a}dulescu \cite{17, 19}):

\smallskip
$H(a):$ $a(y)=a_0(|y|)y$ for all $y\in\RR$ with $a_0(t)>0$ for all $t>0$ and
\begin{itemize}
	\item [(i)] $a_0\in C^1(0,\infty), t\mapsto a_0(t)t$ is strictly increasing on $(0, \infty), a_0(t)t\rightarrow0^+$ as $t\rightarrow0^+$ and
		$$\lim_{t\rightarrow0^+}\frac{a'_0(t)t}{a_0(t)}>-1;$$
	\item [(ii)] $|\nabla a(y)|\leq c_3\frac{k(|y|)}{|y|}$ for all $y\in\RR^\NN\backslash\{0\}$, and for some $c_3>0$;
	\item [(iii)] $(\nabla a(y)\xi,\xi)_{\RR^\NN}\geq\frac{k(|y|)}{|y|}|\xi|^2$ for all $y\in\RR^\NN\backslash\{0\}$, $\xi\in\RR^\NN$;
	\item [(iv)] for $G_0(t)=\int^t_0 a_0(s)s\ ds$ for all $t>0$, we can find $\tau\in(1,p]$ such that
		$$\limsup_{t\rightarrow0^+}\frac{\tau G_0(t)}{t^\tau}\leq c^*,$$
		$$t\mapsto G_0(t^{\frac{1}{\tau}})\ \mbox{is convex.}$$
\end{itemize}

\begin{remark}
Hypotheses $H(a)(i), (ii), (iii)$ were dictated by the nonlinear regularity theory of Lieberman \cite{11} and  the nonlinear maximum principle of Pucci and Serrin \cite[pp. 111, 120]{23}. Hypothesis $H(a)(iv)$ serves the needs of our problem. It is a mild condition and it is satisfied in all cases of interest (see the examples below). Evidently, $G_0(\cdot)$ is strictly convex and strictly increasing. We set $G(y)=G_0(|y|)$ for all $y\in\RR^N$. Then $G\in C^1(\RR^N,\RR),\ G(\cdot)$ is convex, $G(0)=0$, and we have
$$\nabla G(0)=0\ \mbox{and}\ \nabla G(y) = G'_0(|y|)\frac{y}{|y|} = a_0(|y|)y = a(y)\ \mbox{for all}\ y\in\RR^\NN\backslash\{0\}.$$
\end{remark}

Hence $G(\cdot)$ is the primitive of $a(\cdot)$ and so by a well-known property of convex functions, we have
\begin{equation}\label{eq3}
	G(y)\leq(a(y),y)_{\RR^\NN}\ \mbox{for all}\ y\in\RR^\NN.
\end{equation}

The following lemma is an easy consequence of hypotheses $H(a)$ and summarizes the main properties of $a(\cdot)$ (see Papageorgiou and R\u{a}dulescu \cite{17}).

\begin{lemma}\label{lem2}
	If hypotheses $H(a)(i), (ii), (iii)$ hold, then
	\begin{itemize}
		\item [(a)] $y\mapsto a(y)$ is continuous and strictly monotone (thus maximal monotone operator, too);
		\item [(b)] $|a(y)|\leq  c_4(1+|y|^{p-1})$ for all $y\in\RR^\NN$, with $c_4>0$;
		\item [(c)] $(a(y),y)_{\RR^\NN}\geq\frac{c_1}{p-1}|y|^p$ for all $y\in\RR^\NN$ (see (\ref{eq2})).
	\end{itemize}
\end{lemma}

	Then this lemma and (\ref{eq3}) lead to the following growth restrictions for the primitive $G(\cdot)$.

\begin{corollary}\label{c3}
	If hypotheses $H(a)(i), (ii), (iii)$ hold, then $\frac{c_1}{p(p-1)}|y|^p\leq G(y)\leq c_5(1+|y|^p)$ for all $y\in\RR^\NN$ and for some $c_5>0$.
\end{corollary}

Next, we present some characteristic examples of differential operators which fit in the framework provided by hypotheses $H(a)$ (see Papageorgiou and R\u{a}dulescu \cite{17}).

\begin{example}
\begin{itemize}
	\item [(a)] $a(y)=|y|^{p-2}y$ with $1<p<\infty$.
	
		The corresponding differential operator is the $p$-Laplacian defined by
		$$\Delta_p u = {\rm div}\,(|Du|^{p-2}Du)\ \mbox{for all}\ u\in W^{1,p}(\Omega).$$
	\item [(b)] $a(y)=|y|^{p-2}y+|y|^{q-2}y$ with $1<q<p<\infty$.
	
		The corresponding differential operator is the $(p,q)$-Laplacian defined by
		$$\Delta_p u+\Delta_q u\ \mbox{for all}\ u\in W^{1,p}(\Omega).$$
		
		Such operators arise in problems of mathematical physics and recently there have been some existence and multiplicity results for equations driven by such operators, see Cingolani and Degiovanni \cite{2}, Mugnai and Papageorgiou \cite{14}, Papageorgiou and R\u{a}dulescu \cite{15}, Papageorgiou, R\u{a}dulescu and Repov\v{s} \cite{20}, Sun \cite{24}, and Sun, Zhang and Su \cite{25}.
	\item [(c)] $a(y)=(1+|y|^2)^{\frac{p-2}{p}}y$ with $1<p<\infty$.
	
		The corresponding differential operator is the generalized p-mean curvature differential operator defined by
		$${\rm div}\,([1+|Du|^2]^{\frac{p-2}{p}}Du)\ \mbox{for all}\ u\in W^{1,p}(\Omega).$$
		
	\item [(d)] $a(y)=|y|^{p-2}y\left[1+\frac{1}{1+|y|^p}\right]$ with $1<p<\infty$.
	
		The corresponding differential operator is defined by
		$$\Delta_p u + {\rm div}\,\left(\frac{|Du|^{p-2}Du}{1+|Du|^2}\right)\mbox{for all}\ u\in W^{1,p}(\Omega).$$
		
		Such operators arise in problems of plasticity.
\end{itemize}
\end{example}

Now let $A:W^{1,p}(\Omega)\rightarrow W^{1,p}(\Omega)^*$ be the nonlinear map defined by
$$\langle A(u),h\rangle = \int_\Omega(a(Du), Dh)_{\RR^\NN}dz\ \mbox{for all}\ u,h\in W^{1,p}(\Omega).$$

\begin{prop}\label{prop4}
	If hypotheses $H(a)(i), (ii), (iii)$ hold, then $A$ is continuous, monotone (hence maximal monotone, too) and of type $(S)_+$, that is, if $u_n\xrightarrow{w}u$ in $W^{1,p}(\Omega)$ and
	$$\limsup_{n\rightarrow\infty}\langle A(u_n), u_n-u\rangle\leq 0,$$
	then $u_n\rightarrow u$ in $W^{1,p}(\Omega)$
\end{prop}

The following spaces will be used in the analysis of problem (\ref{eq1}):
$$W^{1,p}(\Omega),\ C^{1}(\overline\Omega)\ \mbox{and}\ L^q(\partial\Omega)\ (1\leq q\leq\infty).$$

We denote by $||\cdot||$ the norm of $W^{1,p}(\Omega)$ given by
$$||u|| = \left[||u||^p_p+||Du||^p_p\right]^{\frac{1}{p}}\ \mbox{for all}\ u\in W^{1,p}(\Omega).$$

The Banach space $C^1(\overline\Omega)$ is an ordered Banach space with positive (order) cone given by
$$C_+ = \{u\in C^1(\overline\Omega):u(z)\geq0\ \mbox{for all}\ z\in\overline\Omega\}.$$

This cone has a nonempty interior containing the set
$$D_+ = \{u\in C_+:u(z)>0\ \mbox{for all}\ z\in\overline\Omega\}.$$

On $\partial\Omega$ we consider the $(N-1)$-dimensional Hausdorff (surface) measure $\sigma(\cdot)$. Using this measure, we can define in the usual way the ``boundary" Lebesgue spaces $L^q(\partial\Omega)$ (for $1\leq q\leq\infty$). From the theory of Sobolev spaces we know that there exists a unique continuous linear map $\gamma_0:W^{1,p}(\Omega)\rightarrow L^p(\partial\Omega)$ known as the ``trace map", such that
$$\gamma_0(u) = u|_{\partial\Omega}\ \mbox{for all}\ u\in W^{1,p}(\Omega)\cap C(\overline\Omega).$$

Hence the trace map assigns boundary values to any Sobolev function. The trace map is compact into $L^q(\partial\Omega)$ for all $q\in\left[1, \frac{p(N-1)}{N-p}\right)$ if $p<N$ and for all $q\in[1,+\infty)$ if $p\geq N$. Also, we have
$${\rm im}\,\gamma_0 = W^{-\frac{1}{p'},p}\ (\partial\Omega)(\frac{1}{p}+\frac{1}{p'}=1)\ \mbox{and}\ \ker\gamma_0 = W^{1,p}_0(\Omega).$$

In what follows, for the sake of notational simplicity, we drop the use of the map $\gamma_0$. All restrictions of Sobolev functions on $\partial\Omega$ are understood in the sense of traces.

For $x\in\RR$, let $x^{\pm}=\max\{\pm x,0\}$. Then for any function $u(\cdot)$ we define
$$u^\pm(\cdot)=u(\cdot)^\pm .$$

If $u\in W^{1,p}(\Omega)$, then $u^\pm\in W^{1,p}(\Omega),\ u = u^+-u^-$ and $|u|=u^++u^-$.

Our hypotheses on the potential function $\xi(\cdot)$ and the boundary coefficient $\beta(\cdot)$ are the following:

\smallskip
$\bullet\ $\textbf{$H(\xi)$}: $\xi\in L^\infty(\Omega)$;

\smallskip
$\bullet\ $\textbf{$H(\beta)$}: $\beta\in C^{0,\alpha}(\partial\Omega)$ with $\alpha\in(0,1)$ and $\beta(z)\geq0$ for all $x\in\partial\Omega$.

\begin{remark}\label{rem1}
	If $\beta\equiv0$, then we recover the Neumann problem.
\end{remark}
	
	Let $\gamma:W^{1,p}(\Omega)\rightarrow\RR$ be the $C^1$-functional defined by
	$$\gamma(u) = \int_\Omega G(Du)dz + \int_\Omega\xi(z)|u|^pdz + \int_{\partial\Omega}\beta(z)|u|^pd\sigma\ \mbox{for all}\ u\in W^{1,p}(\Omega).$$
	Also, let $f_0:\Omega\times\RR\rightarrow\RR$ be a Carath\'eodory function such that
	$$|f_0(z,x)|\leq a_0(z)(1+|x|^{r-1})\ \mbox{for almost all}\ z\in\Omega,\ \mbox{and all}\ x\in\RR,$$
	with $a_0\in L^\infty(\Omega)_+, 1<r\leq p^*$ where $p^*=\left\{\begin{array}{ll}
		\frac{Np}{N-p} & \mbox{if}\ p<N \\
		+\infty & \mbox{if}\ p\geq N
	\end{array}\right.$ (the critical Sobolev exponent). Let $F_0(z,x)=\int^x_0f_0(z,s)ds$ and consider the $C^1$-functional $\psi_0:W^{1,p}(\Omega)\rightarrow\RR$ defined by
	$$\psi_0(u)=\frac{1}{p}\gamma(u) - \int_{\Omega} F_0(z,u)dz\ \mbox{for all}\ u\in W^{1,p}(\Omega).$$
	
	The following result is due to Papageorgiou and R\u{a}dulescu \cite{19} and is an outgrowth of the nonlinear regularity theory of Lieberman \cite{11}.

\begin{prop}\label{prop5}
	Assume that hypotheses $H(a)(i), (ii), (iii), H(\xi), H(\beta)$ hold and $u_0\in W^{1,p}(\Omega)$ is a local $C^1(\overline\Omega)$-minimizer of $\psi_0$, that is, there exists $\rho_0>0$ such that
	$$\psi_0(u_0)\leq\psi_0(u_0+h)\ \mbox{for all}\ h\in C^1(\overline\Omega),\ ||h||_{C^1(\overline\Omega)}\leq \rho_0.$$
	Then $u_0\in C^{1,\eta}(\overline\Omega)$ for some $\eta\in(0,1)$ and $u_0$ is also a local $W^{1,p}(\Omega)$-minimizer of $\psi_0$, that is, there exists $\rho_1>0$ such that
	$$\psi_0(u_0)\leq\psi_0(u_0+h)\ \mbox{for all}\ h\in W^{1,p}(\Omega),\ ||h||\leq\rho_1.$$
\end{prop}

In the special case of semilinear equations (that is, when $a(y)=y$ for all $y\in\RR^\NN$), we will be able to improve the multiplicity theorem and produce additional nodal solutions. In this case we can also relax the requirements on the potential function $\xi(z)$ and make use of the spectrum of $u\mapsto -\Delta u+\xi(z)u$ with Robin boundary condition.

So, we consider the following linear eigenvalue problem
\begin{equation}\label{eq4}
	\left\{
	\begin{array}{ll}
		-\Delta u(z) + \xi(z)u(z) = \hat{\lambda}u(z)\ \mbox{in}\ \Omega, \\
		\frac{\partial u}{\partial n} + \beta(z)u = 0\ \mbox{on}\ \partial\Omega.
	\end{array}
	\right\}
\end{equation}

Now we assume that
$$\xi\in L^s(\Omega)\ \mbox{with}\ s>N\ \mbox{and}\ \beta\in W^{1,\infty}(\partial\Omega)\ \mbox{with}\ \beta(z)\geq0\ \mbox{for all}\ x\in\partial\Omega.$$

We consider the $C^1$-functional $\hat{\gamma}:H^1(\Omega)\rightarrow\RR$ defined by
$$\hat{\gamma}(u) = ||Du||^2_2 + \int_\Omega\xi(z)u^2dz + \int_{\partial\Omega}\beta(z)u^2d\sigma\ \mbox{for all}\ u\in H^1(\Omega).$$

From D'Agui, Marano and Papageorgiou \cite{3}, we know that there exists $\mu>0$ such that
\begin{equation}\label{eq5}
	\hat{\gamma}(u) + \mu||u||^2_2 \geq c_6||u||^2\ \mbox{for all}\ u\in H^1(\Omega),\ \mbox{and some}\ c_6>0.
\end{equation}

Using (\ref{eq5}) and the spectral theorem for compact self-adjoint operators on a Hilbert space, we show that the spectrum $\hat{\sigma}(2)$ of (\ref{eq4}) consists of a sequence $\{\hat{\lambda}_k\}_{k\geq1}$ of distinct eigenvalues which satisfy $\hat{\lambda}_k\rightarrow+\infty$ as $k\rightarrow+\infty$. By $E(\hat{\lambda}_k)$ we denote the corresponding eigenspace. We can say the following about these items:
\begin{itemize}
	\item [(i)] $\hat{\lambda}_1$ is simple (that is, $\dim E(\hat{\lambda}_1)=1$) and
		\begin{equation}\label{eq6}
			\hat{\lambda}_1 = \inf\left\{\frac{\hat{\gamma}(u)}{||u||^2_2}:u\in H^1(\Omega),\ u\neq0\right\}.
		\end{equation}
	\item [(ii)] For every $m\geq2$ we have
		\begin{equation}\label{eq7}
			\begin{array}{ll}
				\hat{\lambda}_m & = \inf\left\{\frac{\hat{\gamma}(u)}{||u||^2_2}:u\in \overline{\underset{k\geq m}{\oplus}E(\hat{\lambda}_k)}, u\neq0\right\} \\
				& = \sup\left\{\frac{\hat{\gamma}(u)}{||u||^2_2}:u\in \overset{m}{\underset{k=1}{\oplus}}E(\hat{\lambda}_k), u\neq0\right\}
			\end{array}
		\end{equation}
	\item [(iii)] For every $k\in\NN,\ E(\hat{\lambda}_k)$ is finite dimensional, $E(\hat{\lambda}_k)\subseteq C^1(\overline\Omega)$, and it has the ``Unique Continuation Property" (``UCP" for short), that is, if $u\in E(\hat{\lambda}_k)$ vanishes on a set of positive measure, then $u\equiv0$ (see de Figueiredo and Gossez \cite{5}).
\end{itemize}

In relation (\ref{eq6}), the infimum is realized on $E(\hat{\lambda}_1)$, while in (\ref{eq7}), both the infimum and the supremum are realized on $E(\hat{\lambda}_m)$. Moreover, from the above properties we see that the elements of $E(\hat{\lambda}_1)$ have constant sign, while the elements of $E(\hat{\lambda}_m)$ (for $ m\geq2$) are all nodal (that is, sign changing). By $\hat{u}_1$ we denote the $L^2$-normalized (that is, $||\hat{u}_1||_2=1$) positive eigenfunction corresponding to $\hat{\lambda}_1$. From the regularity theory of Wang \cite{27}, we have that $\hat{u}_1\in C_+$ and using the Harnack inequality (see, for example, Motreanu, Motreanu and Papageorgiou \cite[p. 211]{13}), we have that $\hat{u}_1(z)>0$ for all $z\in\Omega$. Furthermore, if we assume that $\xi^+\in L^\infty(\Omega)$, then $\hat{u}_1\in D_+$ (by the strong maximum principle).

Finally, let us recall some basic definitions and facts from Morse theory (critical groups), which we will need in the sequel.

With $X$ being a Banach space, let $(Y_1, Y_2)$ be a topological pair such that $Y_2\subseteq Y_1\subseteq X$. For every $k\in\NN_0$, let $H_k(Y_1, Y_2)$ denote the $k$th relative singular homology group with integer coefficients for the pair $(Y_1, Y_2)$. For $k<0$, we have $H_k(Y_1, Y_2)=0$.

For $\varphi\in C^1(X,\RR)$ and $c\in\RR$ we introduce the following sets:
\begin{eqnarray*}
	\varphi^c = \{u\in X:\varphi(u)\leq c\}, \\
	K_\varphi = \{u\in X:\varphi'(u)=0\}, \\
	K^c_\varphi = \{u\in K_\varphi:\varphi(u)=c\}.
\end{eqnarray*}

Suppose that $u\in K^c_\varphi$ is isolated. Then the critical groups of $\varphi$ at $u$ are defined by
$$C_k(\varphi, u) = H_k(\varphi^c\cap U, \varphi^c\cap U\backslash\{u\})\ \mbox{for all}\ k\in\NN_0.$$

Here, $U$ is a neighborhood of $u$ such that $K_\varphi\cap\varphi^c\cap U=\{u\}$. The excision property of singular homology theory implies that the above definition of critical groups is independent of the choice of the isolating neighborhood $U$.

Suppose that $\varphi\in C^1(X,\RR)$ satisfies the C-condition and that $\inf\varphi(K_\varphi)>-\infty$. Then the critical groups of $\varphi$ at infinity are defined by
$$C_k(\varphi, \infty) = H_k(X, \varphi^c)\ \mbox{for all}\ k\in\NN_0,\ \mbox{with}\ c<\inf\varphi(K_\varphi).$$

This definition is independent of the choice of $c<\inf\varphi(K_\varphi)$. To see this, let $c'<\inf\varphi(K_\varphi)$ and without any loss of generality assume that $c'<c$. Then from Motreanu, Motreanu and Papageorgiou \cite[Theorem 5.34, p. 110]{13},  we have that
\begin{eqnarray*}
	&& \varphi^{c'}\ \mbox{is a strong deformation retract of}\ \varphi^c, \\
	&\Rightarrow & H_k(X,\varphi^c) = H_k(X,\varphi^{c'})\ \mbox{for all}\ k\in\NN_0 \\
	&& \mbox{(see Motreanu, Motreanu and Papageorgiou \cite[Corollary 6.15, p. 145]{13})}.
\end{eqnarray*}

Assume that $K_\varphi$ is finite. We introduce the following quantities
\begin{eqnarray*}
	M(t, u) & = & \underset{k\in\NN_0}{\sum}{\rm rank}\, C_k(\varphi,u)t^k\ \mbox{for all}\ t\in\RR, u\in K_\varphi, \\
	P(t, \infty) & = & \underset{k\in\NN_0}{\sum}{\rm rank}\, C_k(\varphi, \infty)t^k\ \mbox{for all}\ t\in\RR.
\end{eqnarray*}

The Morse relation says that
\begin{equation}\label{eq8}
	\underset{u\in K_\varphi}{\sum} M(t,u) = P(t,\infty) + (1+t)Q(t),
\end{equation}
where $Q(t)=\underset{k\in\NN_0}{\sum}\beta_kt^k$ is a formal series in $t\in\RR$ with nonnegative integer coefficients.

Let $H$ be a Hilbert space, $u\in H$, and $U$ a neighborhood of $u$. Suppose that $\varphi\in C^2(U)$. If $u\in K_\varphi$, then the ``Morse index" $m$ of $u$ is defined to be the supremum of the dimensions of the vector subspaces of $H$ on which $\varphi''(u)$ is negative definite. The ``nullity" $\nu$ of $u$ is the dimension of $\ker\varphi''(u)$. We say that $u\in K_\varphi$ is ``nondegenerate" if $\varphi''(u)$ is invertible (that is, $\nu=0$). Suppose that $\varphi\in C^2(U)$ and $u\in K_\varphi$ is isolated and nondegenerate with Morse index $m$. Then
$$C_k(\varphi, u) = \delta_{k,m}\ZZ\ \mbox{for all}\ k\in\NN_0.$$

Here $\delta_{k,m}$ denotes the Kronecker symbol, that is,
$$\delta_{k,m} = \left\{\begin{array}{ll}
	1 & \mbox{if}\ k=m \\
	0 & \mbox{if}\ k\neq m.
\end{array}\right.$$

\section{Solutions of Constant Sign}

In this section, we produce solutions of constant sign for problem (\ref{eq1}). We assume the following conditions on the reaction term $f(z,x)$.

\smallskip
\textbf{$H(f)_1$}: $f:\Omega\times\RR\rightarrow\RR$ is a Carath\'eodory function such that $f(z,0)=0$ for almost all $z\in\Omega$ and
\begin{itemize}
	\item [(i)] there exist $\eta>0$ and $a_\eta\in L^\infty(\Omega)_+$ such that
		$$|f(z,x)|\leq a_\eta(z)\ \mbox{for almost all}\ z\in\Omega,\ \mbox{and all}\ x\in[-\eta, \eta];$$
	\item [(ii)] if $F(z,x)=\int^x_0f(z,s)ds$, then there exist $\eta_0>0, q\in(1,\tau)$ ($\tau>1$ as in hypothesis $H(a)(iv)$) and $\delta_0>0$ such that
		$$\eta_0|x|^q\leq f(z,x)x\leq qF(z,x)\ \mbox{for almost all}\ z\in\Omega,\ \mbox{and all}\ |x|\leq\delta_0;$$
	\item [(iii)] with $\eta>0$ as in $(i)$ we have
		$$f(z,\eta) - \xi(z)\eta^{p-1}  \leq 0 \leq f(z, -\eta) + \xi(z)\eta^{p-1}\ \mbox{for almost all}\ z\in\Omega.$$
\end{itemize}
\begin{remark}
	We see that no global growth condition is imposed on $f(z,\cdot)$. All our hypotheses on $f(z,\cdot)$ concern its behaviour near zero. Note that $H(f)_1,\, (ii),\, (iii)$ imply a kind of oscillatory behaviour near zero for $x\mapsto f(z,x)-\xi(z)|x|^{p-2}x$.
\end{remark}

Evidently, we can find $\vartheta_0>0$ such that
\begin{equation}\label{eq9}
	f(z,x)x \geq \eta_0|x|^q - \vartheta_0|x|^p\ \mbox{for almost all}\ z\in\Omega,\ \mbox{and all}\ |x|\leq\eta .
\end{equation}

Then we define
\begin{equation}\label{eq10}
	\mu(z,x)=\left\{\begin{array}{ll}
		-\eta_0 \eta^{q-1} + \vartheta_0\eta^{p-1} & \mbox{if}\ x<-\eta \\
		\eta_0|x|^{q-2}x-\vartheta_0|x|^{p-2}x & \mbox{if}\ -\eta\leq x\leq\eta \\
		\eta_0\eta^{q-1}-\vartheta_0\eta^{p-1} & \mbox{if}\ \eta<x.
	\end{array}\right.
\end{equation}

Note that $\mu(z,x)$ is a Carath\'eodory function and for all $z\in\Omega,\ \mu(z,\cdot)$ is odd. We consider the following auxiliary Robin problem
\begin{equation}\label{eq11}
	\left\{\begin{array}{ll}
		-{\rm div}\, a(Du(z)) + ||\xi^+||_\infty|u(z)|^{p-2}u(z)=\mu(z,u(z))\quad \mbox{in}\ \Omega, \\
		\frac{\partial u}{\partial n_a} + \beta(z)|u|^{p-2}u=0\quad \mbox{on}\ \partial\Omega.
	\end{array}\right\}
\end{equation}

In what follows, given $h_1, h_2\in W^{1,p}(\Omega)$, we set
$$[h_1, h_2] = \{u\in W^{1,p}(\Omega):h_1(z)\leq u(z)\leq h_2(z)\ \mbox{for almost all}\ z\in\Omega\}.$$
\begin{prop}\label{prop6}
	If hypotheses $H(a), H(\xi), H(\beta)$ hold, then problem (\ref{eq10}) admits a unique positive solution
	$$\tilde{u}\in[0,\eta]\cap D_+$$
	and since (\ref{eq10}) is odd, then $\tilde{v} = -\tilde{u}\in[-\eta,0]\cap D_+$ is the unique negative solution of (\ref{eq10}).
\end{prop}
\begin{proof}
	We first show the existence of a positive solution. So, let $\vartheta>0$ be such that
	$$\eta_0\leq[||\xi^+||_\infty+\vartheta]\eta^{p-1}.$$
	
	We introduce the following Carath\'eodory function
	\begin{equation}\label{eq12}
		\hat{\mu}_+(z,x)=\left\{\begin{array}{ll}
			0 & \mbox{if}\ x<0 \\
			\mu(z,x) + \vartheta x^{p-1} & \mbox{if}\ 0\leq x\leq\eta \\
			\mu(z,\eta) + \vartheta\eta^{p-1} & \mbox{if}\ \eta<x.
		\end{array}\right.
	\end{equation}
	
We set $\hat{M}_+(z,x) = \int^x_0\hat{\mu}_+(z,s)ds$ and consider the $C^{1}$-functional $\psi_+:W^{1,p}(\Omega)\rightarrow\RR$ defined by
\begin{eqnarray*}
	&&\psi_+(u) = \int_\Omega G(Du)dz + \frac{||\xi^+||_\infty+\vartheta}{p}||u||^p_p + \frac{1}{p}\int_{\partial\Omega}\beta(z)|u|^pd\sigma - \int_\Omega\hat{M}_+(z,u)dz\\
	&&\mbox{for all}\ u\in W^{1,p}(\Omega).
\end{eqnarray*}

Corollary \ref{c3}, hypothesis $H(\beta)$ and (\ref{eq12}) imply that
$$\psi_+\ \mbox{is coercive}.$$

Also, using the Sobolev embedding theorem and the compactness of the trace map, we see that $\psi_+$ is sequentially weakly lower semicontinuous. Invoking the Weierstrass-Tonelli theorem, we can find $\tilde{u}\in W^{1,p}(\Omega)$ such that
\begin{equation}\label{eq13}
	\psi_+(\tilde{u}) = \inf\{\psi_+(u):u\in W^{1,p}(\Omega)\}.
\end{equation}

Hypothesis $H(a)(iv)$ implies that we can find $c_7>0$ and $\delta\in(0,\delta_0]$ such that
\begin{equation}\label{eq14}
	G(y)\leq c_7|y|^{\tau}\ \mbox{for all}\ y\in\RR^N\ \mbox{with}\ |y|\leq\delta\,.
\end{equation}

Let $u\in D_+$ and choose small $t\in(0,1)$ such that $tu\leq\delta_0$. Then we have
\begin{eqnarray*}
	&&\psi_+(tu)\leq t^{\tau}||Du||^\tau_\tau + c_8t^p||u||^p_p - c_9t^q||u||^q_q\\
	&&\mbox{for some}\ c_8,c_9>0\ \mbox{(see (\ref{eq12}), (\ref{eq14}) and hypothesis $H(\beta)$)}.
\end{eqnarray*}

Recall that $1<q<\tau\leq p$. So, by choosing $t\in(0,1)$ even smaller if necessary, we have
$$\begin{array}{ll}
	& \psi_+(tu)<0, \\
	\Rightarrow & \psi_+(\tilde{u})<0=\psi_+(0)\ \mbox{(see (\ref{eq13}))}, \\
	\Rightarrow & \tilde{u}\neq0.
\end{array}$$

From (\ref{eq13}) we have
\begin{eqnarray}\label{eq15}
	& \psi'_+(\tilde{u})=0 \nonumber \\
	\Rightarrow & \langle A(\tilde{u}, h)\rangle + (||\xi^+||_\infty+\vartheta)\int_\Omega|\tilde{u}|^{p-2}\tilde{u}hdz + \int_{\partial\Omega}\beta(z)|\tilde{u}|^{p-2}\tilde{u}hd\sigma = \\
	& \int_\Omega\hat{\mu}_+(z,\tilde{u})hdz\ \mbox{for all}\ h\in W^{1,p}(\Omega). \nonumber
\end{eqnarray}

In (\ref{eq15}) we choose $h=-\tilde{u}^{-}\in W^{1,p}(\Omega)$. Using Lemma \ref{lem2}, we have
\begin{eqnarray*}
	&& \frac{c_1}{p-1}||D\tilde{u}^-||^p_p + [||\xi^+||_\infty+\vartheta] ||\tilde{u}^{-}||^p_p\leq0 \\
	&\Rightarrow & \tilde{u}\geq0,\ \tilde{u} \neq0.
\end{eqnarray*}

Also in (\ref{eq15}) we choose $h=(\tilde{u}-\eta)^+\in W^{1,p}(\Omega)$. Then
\begin{eqnarray*}
	&& \langle A(\tilde{u}), (\tilde{u}-\eta)^+\rangle + [||\xi^+||_\infty+\vartheta]\int_\Omega\tilde{u}^{p-1}(\tilde{u}-\eta)^+dz + \int_{\partial\beta}\beta(z)\tilde{u}^{p-1}(\tilde{u}-\eta)^+d\sigma \\
	&= & \int_\Omega[\eta_0\eta^{q-1}-(\vartheta_0-\vartheta)\eta^{p-1}](\tilde{u}-\eta)^+dz\ (\mbox{see (\ref{eq12}), (\ref{eq10})}) \\
	&\leq & \langle A(\eta), (\tilde{u}-\eta)^+\rangle + [||\xi^+||_\infty+\vartheta]\int_\Omega\eta^{p-1}(\tilde{u}-\eta)^+dz + \int_{\partial\Omega}\beta(z)\tilde{u}^{p-1}(\tilde{u}-\eta)^+d\sigma \\
	&& \mbox{(recall that $\eta_0\leq||\xi^+||_\infty\eta^{p-q}$ and see hypothesis $H(\beta)$)}, \\
	\Rightarrow& & \langle A(\tilde{u}) - A(\eta),(\tilde{u}-\eta)^+ \rangle + [||\xi^+||_\infty + \vartheta]\int_\Omega(\tilde{u}^{p-1}-\eta^{p-1})(\tilde{u}-\eta)^+dz\leq0, \\
	\Rightarrow && \tilde{u}\leq\eta.
\end{eqnarray*}

So, we have proved that
\begin{equation}\label{eq16}
	\tilde{u}\in[0,\eta],\ \tilde{u}\neq0.
\end{equation}

Using (\ref{eq10}), (\ref{eq12}) and (\ref{eq16}) in (\ref{eq15}), we obtain
\begin{eqnarray}\label{eq17}
	&& \langle A(\tilde{u}),h\rangle+||\xi^+||_\infty\int_\Omega\tilde{u}^{p-1}hdz + \int_{\partial\Omega}\beta(z)\tilde{u}^{p-1}hd\sigma = \int_\Omega[\eta_0\tilde{u}^{q-1}-\vartheta_0\tilde{u}^{p-1}]hdz\nonumber\\
	&&\mbox{for all}\ h\in W^{1,p}(\Omega), \nonumber \\
	&\Rightarrow & -{\rm div}\,a(D\tilde{u}(z)) + ||\xi^+||_\infty\tilde{u}(z)^{p-1} = \eta_0\tilde{u}(z)^{q-1}-\vartheta_0\tilde{u}(z)^{p-1}\ \mbox{for almost all}\ z\in\Omega, \nonumber \\
	&& \frac{\partial\tilde{u}}{\partial n_a} + \beta(z)\tilde{u}^{p-1}=0\ \mbox{on}\ \partial\Omega\ \mbox{(see Papageorgiou and R\u{a}dulescu \cite{16})}.
\end{eqnarray}

From (\ref{eq17}) and Papageorgiou and R\u{a}dulescu \cite{19}, we have
$$\tilde{u}\in L^\infty(\Omega).$$

Then from the regularity theory of Lieberman \cite{11} we have
$$\tilde{u}\in C_+\backslash\{0\}.$$

From (\ref{eq17}) we have
$${\rm div}\,a(D\tilde{u}(z))\leq[||\xi^+||_{\infty}+\vartheta_0]\tilde{u}(z)^{p-1}\ \mbox{for almost all}\ z\in\Omega.$$

Hence by the nonlinear strong maximum principle of Pucci and Serrin \cite[pp. 111, 120]{23}, we have
$$\tilde{u}\in D_+.$$

Next, we show that this positive solution is unique. To this end, we introduce the integral functional $j:L^1(\Omega)\rightarrow\overline{\RR}=\RR\cup\{+\infty\}$ defined by
$$j(u)=\left\{\begin{array}{ll}
	\int_\Omega G(Du^{\frac{1}{\tau}})dz + \frac{||\xi^+||_\infty}{p}||u||^{\frac{p}{\tau}}_{\frac{p}{\tau}} + \frac{1}{p}\int_{\partial\Omega}\beta(z)u^{\frac{p}{\tau}}d\sigma\  & \mbox{if}\ u\geq0,\ u^{\frac{1}{\tau}}\in W^{1,p}(\Omega) \\
		+\infty\ & \mbox{otherwise}.
\end{array}\right.$$

Suppose that $u_1,u_2\in {\rm dom}\, j=\{u\in L^1(\Omega):j(u)<\infty\}$ (the effective domain of $j(\cdot)$).

Let $y_1=u_1^{\frac{1}{\tau}}, y_2=u_2^{\frac{1}{\tau}}$. Then $y_1, y_2\in W^{1,p}(\Omega)$. We set
$$y=[tu_1 + (1-t)u_2]^{\frac{1}{\tau}}\ \mbox{for every}\ t\in[0,1].$$

We have $y\in W^{1,p}(\Omega)$. Using Lemma 1 of Diaz and Saa \cite{4}, we have
\begin{eqnarray*}
	|Dy(z)| & \leq & [t|Dy_1(z)|^\tau + (1-t)|Dy_2(z)|^\tau]^{\frac{1}{\tau}}, \\
	\Rightarrow G_0(|Dy(z)|) & \leq & G_0([t|Dy_1(z)|^\tau] + (1-t)|Dy_2(z)|^\tau)^{\frac{1}{\tau}}\ \mbox{(since $G_0(\cdot)$ is increasing)} \\
	& \leq & tG_0(|Dy_1(z)|) + (1-t)G_0(|Dy_2(z)|)\ \mbox{for almost all}\ z\in\Omega\ \\
	& & \mbox{(see hypothesis $H(a)(iv)$)} \\
	\Rightarrow G(Dy(z)) & \leq & tG(Du_1(z)^{\frac{1}{\tau}}) + (1-t)G(Du_2(z)^{\frac{1}{\tau}})\ \mbox{for almost all}\ z\in\Omega, \\
	\Rightarrow & u\mapsto &\int_\Omega G(Du^{\frac{1}{\tau}})dz\ \mbox{is convex}.
\end{eqnarray*}

Also since $\tau\leq p$ and $\beta\geq0$ (see hypothesis $H(\beta)$), it follows that
$${\rm dom}\, j\ni u\mapsto\frac{||\xi^+||_\infty}{p}||u||^{\frac{p}{\tau}}_{\frac{p}{\tau}} + \frac{1}{p}\int_{\partial\Omega}\beta(z)u^{\frac{p}{\tau}}dz\ \mbox{is convex}.$$

It follows that the integral functional $j(\cdot)$ is convex and, by Fatou's lemma, it is lower semicontinuous.

Suppose that $\tilde{u}, \hat{u}\in W^{1,p}(\Omega)$ are two positive solutions of the auxiliary problem (\ref{eq11}). From the first part of the proof we have
\begin{equation}\label{eq18}
	\tilde{u},\hat{u}\in [0,\eta]\cap D_+.
\end{equation}

Therefore for every $h\in C_1(\overline{\Omega})$ and for $|t|$ small, we have
$$\tilde{u}+th\in {\rm dom}\, j\ \mbox{and}\ \hat{u}+th\in {\rm dom}\,j.$$

Because of the convexity of $j(\cdot)$, we see that $j(\cdot)$ is G\^ateaux differentiable at $\tilde{u}^\tau$ and at $\hat{u}^\tau$ in the direction $h$. Using the chain rule and the nonlinear Green's identity (see Gasinski and Papageorgiou \cite[p. 210]{6}), we get
\begin{eqnarray*}
	j'(\tilde{u}^\tau)(h) & = & \frac{1}{\tau}\int_\Omega\frac{-{\rm div}\, a(D\tilde{u}) + ||\xi^+||_\infty\tilde{u}^{p-1}}{\tilde{u}^{\tau-1}}hdz \\
	j'(\hat{u}^\tau)(h) & = & \frac{1}{\tau}\int_\Omega\frac{-{\rm div}\, a(D\hat{u}) + ||\xi^+||_\infty\tilde{u}^{p-1}}{\hat{u}^{\tau-1}}hdz\ \mbox{for all}\ h\in W^{1,p}(\Omega).
\end{eqnarray*}

Recall that $j(\cdot)$ is convex, hence $j'(\cdot)$ is monotone. Hence we have
\begin{eqnarray}\label{eq19}
	0 & \leq & \int_\Omega\left[\frac{-{\rm div}\, a(D\tilde{u}) + ||\xi^+||_\infty\tilde{u}^{p-1}}{\tilde{u}^{\tau-1}} - \frac{-{\rm div}\, a(D\hat{u}) + ||\xi^+||_\infty\hat{u}^{p-1}}{\hat{u}^{\tau-1}}\right](\tilde{u}^\tau-\hat{u}^{\tau})dz \nonumber \\
	& = & \int_\Omega\left[\frac{\mu(z,\tilde{u})}{\tilde{u}^{\tau-1}} - \frac{\mu(z,\hat{u})}{\hat{u}^{\tau-1}}\right](\tilde{u}^\tau-\hat{u}^\tau)dz\ \mbox{(see (\ref{eq11}))} \nonumber \\
	& = & \int_\Omega\left(\eta_0\left[\frac{1}{\tilde{u}^{\tau-q}} - \frac{1}{\hat{u}^{\tau-q}}\right] - \vartheta_0\left[\tilde{u}^{p-\tau}-\hat{u}^{p-\tau}\right]\right)(\tilde{u}^\tau-\hat{u}^\tau)dz\ \mbox{(see (\ref{eq10}) and (\ref{eq18}))}.
\end{eqnarray}

By hypothesis $q<\tau\leq p$. So, from (\ref{eq19}) we infer that
$$\tilde{u} = \hat{u}.$$

This proves the uniqueness of the positive solution
$$\tilde{u}\in[0,\eta]\cap D_+.$$

Equation (\ref{eq11}) is odd. So, it follows that
$$\tilde{v}=-\tilde{u}\in[-\eta,0]\cap(-D_+)$$
is the unique negative solution of (\ref{eq11}).
\end{proof}

Next, we produce constant sign solutions for (\ref{eq1}). For this purpose we introduce the sets
$$S_+ = \mbox{the set of positive solutions for problem (\ref{eq11}) in the order interval}\ [0,\eta],$$
$$S_- = \mbox{the set of negative solutions for problem (\ref{eq11}) in the order interval}\ [-\eta, 0].$$
\begin{prop}\label{prop7}
	If hypotheses $H(a), H(\xi), H(\beta), H(f)_1$ hold, then $\emptyset\neq S_+\subseteq D_+$ and $\emptyset\neq S_-\subseteq -D_+$.
\end{prop}

\begin{proof}
	Let $\eta>0$ be as in hypothesis $H(f)_1(iii)$ and fix $\vartheta>||\xi||_{\infty}$ (see hypothesis $H(\xi)$). We introduce the following Carath\'eodory function
	\begin{equation}\label{eq20}
		\hat{f}_+(z,x)=\left\{\begin{array}{ll}
			0 & \mbox{if}\ x<0 \\
			f(z,x)+\vartheta x^{p-1} & \mbox{if}\ 0\leq x\leq\eta \\
			f(z,\eta) + \vartheta\eta^{p-1} & \mbox{if}\ \eta<x.
		\end{array} \right.
	\end{equation}
	
	We set $\hat{F}_+(z,x)=\int^x_0\hat{f}_+(z,s)ds$ and consider the $C^1$-functional $\hat{\varphi}_+:W^{1,p}(\Omega)\rightarrow\RR$ defined by
	\begin{eqnarray*}
		&&\hat{\varphi}_+(u)=\int_\Omega G(Du)dz  + \frac{1}{p}\int_\Omega(\xi(z)+\vartheta)|u|^pdz + \frac{1}{p}\int_{\partial\Omega}\beta(z)|u|^pd\sigma - \int_\Omega\hat{F}_+(z,u)dz\\
		&&\mbox{for all}\ u\in W^{1,p}(\Omega).
	\end{eqnarray*}
	
	Using Lemma \ref{lem2}, the fact that $\vartheta>||\xi||_\infty$, hypothesis $H(\beta)$ and (\ref{eq20}), we see that
	$$\hat{\varphi}_+\ \mbox{is coercive}.$$
	
	Also, $\hat{\varphi}_+$ is sequentially weak lower semicontinuous. So, by the Weierstrass-Tonelli theorem, we can find $u_0\in W^{1,p}(\Omega)$ such that
	\begin{equation}\label{eq21}
		\hat{\varphi}_+(u_0)=\inf \left\{\hat{\varphi}_+(u):u\in W^{1,p}(\Omega)\right\}.
	\end{equation}
	
	Hypothesis $H(f)_1(ii)$ implies that
	$$\begin{array}{ll}
		& \hat{\varphi}_+(u_0)<0=\hat{\varphi}_+(0)\ \mbox{(see the proof of Proposition \ref{prop6})} \\
		\Rightarrow & u_0\neq0.
	\end{array}$$
	
	From (\ref{eq21}) we have
	\begin{eqnarray}\label{eq22}
		&& \hat{\varphi}'_+(u_0)=0, \\
		&\Rightarrow & \langle A(u_0, h)\rangle + \int_\Omega(\xi(z)+\vartheta)|u_0|^{p-2}u_0hdz + \int_{\partial\Omega}\beta(z)|u_0|^{p-2}u_0 hd\sigma = \int_\Omega\hat{f}_+(z,u_0)hdz \nonumber \\
		&& \mbox{for all}\ h\in W^{1,p}(\Omega).\nonumber
	\end{eqnarray}
	
	In (\ref{eq22}) we choose $h=-u^-_0\in W^{1,p}(\Omega)$. Then
	$$\begin{array}{ll}
		& \frac{c_1}{p-1}||Du_0^-||^p_p + c_{10}||u^-_0||^p_p \leq 0\ \mbox{for some}\ c_{10}>0 \\
		& \mbox{(see Lemma \ref{lem2}, (\ref{eq20}), hypothesis $H(\beta)$ and recall that $\vartheta>||\xi||_\infty$)} \\
		\Rightarrow & u_0\geq0,\ u_0\neq0.
	\end{array}$$
	
	Next, in (\ref{eq22}) we choose $h=(u_0-\eta)^+\in W^{1,p}(\Omega)$. Then
	\begin{eqnarray*}
		&& \langle A(u_0), (u_0-\eta)^+\rangle + \int_\Omega(\xi(z)+\vartheta)u^{p-1}_0(u_0-\eta)^+dz + \int_{\partial\Omega} \beta(z)u_0^{p-1}(u_0-\eta)^+d\sigma \\
		&= & \int_\Omega[f(z,\eta)+\vartheta\eta^{p-1}](u_0-\eta)^+dz\ \mbox{(see (\ref{eq20}))} \\
		&\leq & \langle A(\eta), (u_0-\eta)^+\rangle + \int_\Omega(\xi(z)+\vartheta)\eta^{p-1}(u_0-\eta)^+dz + \int_{\partial\Omega}\beta(z)u_0^{p-1}(u_0-\eta)^+d\sigma \\
		&& \mbox{(note that $A(\eta)=0$ and see hypotheses $H(f)_1(iii), H(\beta)$),} \\
		\Rightarrow && \langle A(u_0)-A(\eta), (u_0-\eta)^+\rangle + \int_\Omega(\xi(z)+\vartheta)(u_0^{p-1}-\eta^{p-1})(u_0-\eta)^+dz\leq0, \\
		\Rightarrow && u_0\leq\eta\ \mbox{(recall that $\vartheta>||\xi||_\infty$)}.
	\end{eqnarray*}
	
	So, we have proved that
	\begin{equation}\label{eq23}
		u_0\in[0,\eta],\ u_0\neq0.
	\end{equation}
	
	On account of (\ref{eq20}) and (\ref{eq23}), equation (\ref{eq22}) becomes
	\begin{eqnarray}\label{eq24}
		&& \langle A(u_0),h \rangle + \int_\Omega\xi(z)u^{p-1}_0hdz + \int_{\partial\Omega}\beta(z)u_0^{p-1}hdz\sigma = \int_\Omega f(z,u_0)hdz\nonumber\\
		&&\mbox{for all}\ h\in W^{1,p}(\Omega), \nonumber \\
		&\Rightarrow & -{\rm div}\,a(Du_0(z)) + \xi(z)u_0(z)^{p-1} = f(z,u_0(z))\ \mbox{for almost all}\ x\in\Omega , \nonumber \\
		&& \frac{\partial u_0}{\partial n_a} + \beta(z)u_0^{p-1}=0\ \mbox{on}\ \partial\Omega\ \mbox{(see Papageorgiou and R\u{a}dulescu \cite{16})}.
	\end{eqnarray}
	
	From (\ref{eq23}), (\ref{eq24}), hypothesis $H(f)_1(i)$ and Papageorgiou and R\u{a}dulescu \cite{19}, we have
	$$u_0\in L^\infty(\Omega).$$
	
	Next, applying the nonlinear regularity theory of Lieberman \cite{11}, we have
	$$u_0\in C_+\backslash\{0\}.$$
	
	Hypotheses $H(f)_1(i), (ii)$ imply that we can find $c_{11}>0$ such that
	$$f(z,x)+c_{11}x^{p-1}\geq0\ \mbox{for almost all}\ x\in\Omega,\ \mbox{and all}\ 0\leq x\leq\eta.$$
	
	Then (\ref{eq23}) and (\ref{eq24}) imply that
	\begin{eqnarray*}
		& {\rm div}\,a(Du_0(z))\leq[||\xi||_\infty + c_{11}] u_0(z)^{p-1}\ \mbox{for almost all}\ z\in\Omega, \\
		\Rightarrow & u_0\in D_+\ \mbox{(see Pucci and Serrin \cite[pp. 111, 120]{23})}.
	\end{eqnarray*}
	
	Therefore we have proved that $\emptyset\neq S_+\subseteq D_+$.
	
	For negative solutions we consider the Carath\'eodory function
	\begin{equation}\label{eq25}
		\hat{f}_-(z,x)=\left\{\begin{array}{ll}
			f(z,-\eta)-\vartheta\eta^{p-1} & \mbox{if}\ x<0 \\
			f(z,x)+\vartheta|x|^{p-2}x & \mbox{if}\ -\eta\leq x\leq0 \\
			0 & \mbox{if}\ 0<x
		\end{array}\right.
	\end{equation}
	(recall that $\vartheta>||\xi||_\infty$). Let $\hat{F}_-(z,x)=\int^x_0\hat{f}_-(z,s)ds$ and consider the $C^1$-functional $\hat{\varphi}_-:W^{1,p}(\Omega)\rightarrow\RR$ defined by
	\begin{eqnarray*}
		&&\hat{\varphi}_-(u) = \int_\Omega G(Du)dz + \frac{1}{p}\int_\Omega(\xi(z)+\vartheta)|u|^p dz + \frac{1}{p}\int_{\partial\Omega}\beta(z)|u|^pd\sigma - \int_\Omega \hat{F}_-(z,u)dz\\
		&&\mbox{for all}\ u\in W^{1,p}(\Omega).
	\end{eqnarray*}
	
	Reasoning as above, using this time $\hat{\varphi}_-$ and (\ref{eq25}), we produce a negative solution
	$$v_0\in[-\eta, 0]\cap (-D_+).$$
	
	Therefore $\emptyset\neq S_-\subseteq -D_+$.
\end{proof}

The next result provides a lower bound for the elements of $S_+$ and an upper bound for the elements of $S_-$. These bounds will lead to the existence of extremal constant sign solutions.
\begin{prop}\label{prop8}
	If hypotheses $H(a),H(\xi),H(\beta),H(f)_1$ hold, then $\tilde{u}\leq u$ for all $u\in S_+$ and $v\leq\tilde{v}$ for all $v\in S_-.$
\end{prop}
\begin{proof}
	Let $u\in S_+$ and consider the following Carath\'eodory function
	\begin{eqnarray}\label{eq26}
		\gamma_+(z,x)=\left\{\begin{array}{ll}
			0&\mbox{if}\ x<0\\
			\eta_0 x^{q-1}-(\vartheta_0-\hat{\vartheta})x^{p-1}&\mbox{if}\ 0\leq x\leq u(z),\\
			\eta_0u(z)^{q-1}-(\vartheta_0-\hat{\vartheta})u(z)^{p-1}&\mbox{if}\ u(z)<x
		\end{array}\right.\ \ \hat{\vartheta}>0.
	\end{eqnarray}
	
	We set $\Gamma_+(z,x)=\int^x_0\gamma_+(z,s)ds$ and consider the $C^1$-functional $\hat{k}_+:W^{1,p}(\Omega)\rightarrow\RR$ defined  by
	\begin{eqnarray*}
		&&\hat{k}_+(u)=\int_{\Omega}G(Du)dz+\frac{1}{p}[||\xi^+||_{\infty}+\hat{\vartheta}]||u||^p_p+\int_{\partial\Omega}\beta(z)|u|^pd\sigma-\int_{\Omega}\Gamma_+(z,u)dz\\
		&&\mbox{for all}\ u\in W^{1,p}(\Omega).
	\end{eqnarray*}
	
	Evidently, $\hat{k}_+$ is coercive (see Lemma \ref{lem2}, (\ref{eq26}) and recall that $\hat{\vartheta}>0$). Also, it is sequentially weakly lower semicontinuous. So, we can find $\hat{u}\in W^{1,p}(\Omega)$ such that
	\begin{equation}\label{eq27}
		\hat{k}_+(\hat{u})=\inf\{\hat{k}_+(u):u\in W^{1,p}(\Omega)\}.
	\end{equation}
	
	By (\ref{eq26}) and since $q<\tau\leq p$, we have
	\begin{eqnarray*}
		&&\hat{k}_+(\hat{u})<0=\hat{k}_+(0)\ (\mbox{see the proof of Proposition \ref{prop6}}),\\
		&\Rightarrow&\hat{u}\neq 0.
	\end{eqnarray*}
	
	From (\ref{eq27}) we have
	\begin{eqnarray}\label{eq28}
		&&\hat{k}'_+(\hat{u})\neq 0,\nonumber\\
		&\Rightarrow&\left\langle A(\hat{u},h)\right\rangle+[||\xi^+||_{\infty}+\hat{\vartheta}]\int_{\Omega}|\hat{u}|^{p-2}\hat{u}hdz+\int_{\partial\Omega}\beta(z)|\hat{u}|^{p-2}\hat{u}hdz=\int_{\Omega}\gamma_+(z,\hat{u})hdz\\
		&&\mbox{for all}\ h\in W^{1,p}(\Omega).\nonumber
	\end{eqnarray}
	
	In (\ref{eq28}) we choose $h=-\hat{u}^-\in W^{1,p}(\Omega)$. We obtain
	$$\hat{u}\geq 0,\ \hat{u}\neq 0.$$
	
	Also, in (\ref{eq28}) we choose $h=(\hat{u}-u)^+\in W^{1,p}(\Omega)$. Then
	\begin{eqnarray*}
		&&\left\langle A(\hat{u}),(\hat{u}-u)^+\right\rangle+[||\xi^+||_{\infty}+\hat{\vartheta}]\int_{\Omega}\hat{u}^{p-1}(\hat{u}-u)^+dz+\int_{\partial\Omega}\beta(z)\hat{u}^{p-1}(\hat{u}-u)^+d\sigma\\
		&=&\int_{\Omega}[\eta_0u^{q-1}-\vartheta_0u^{p-1}](\hat{u}-u)^+dz+\hat{\vartheta}\int_{\Omega}u^{p-1}(\hat{u}-u)^+dz\ (\mbox{see (\ref{eq26})})\\
		&\leq&\int_{\Omega}[f(z,u)+\hat{\vartheta}u^{p-1}](\hat{u}-u)^+dz\ (\mbox{see (\ref{eq9})})\\
		&=&\left\langle A(u),(\hat{u}-u)^+\right\rangle+\int_{\Omega}(\xi(z)+\hat{\vartheta})u^{p-1}(\hat{u}-u)^+dz+\int_{\partial\Omega}\beta(z)u^{p-1}(\hat{u}-u)^+d\sigma\\
		&&(\mbox{since}\ u\in S_+)\\
		&\leq&\left\langle A(u),(\hat{u}-u)^+\right\rangle+[||\xi^+||_{\infty}+\hat{\vartheta}]\int_{\Omega}u^{p-1}(\hat{u}-u)^+dz+\int_{\partial\Omega}\beta(z)u^{p-1}(\hat{u}-u)^+d\sigma,\\
		&\Rightarrow&\left\langle A(\hat{u})-A(u),(\hat{u}-u)^+\right\rangle+[||\xi^+||_{\infty}+\hat{\vartheta}]\int_{\Omega}(\hat{u}^{p-1}-u^{p-1})(\hat{u}-u)^+dz\leq 0\\
		&&(\mbox{see hypothesis}\ H(\beta)),\\
		&\Rightarrow&\hat{u}\leq u.
	\end{eqnarray*}
	
	From (\ref{eq26}) we see that $\hat{u}$ is a positive solution of (\ref{eq11}) and so
	\begin{eqnarray*}
		&&\hat{u}=\tilde{u}\in D_+\ (\mbox{see Proposition \ref{prop6}})\\
		&\Rightarrow&\tilde{u}\leq u\ \mbox{for all}\ u\in S_+.
	\end{eqnarray*}
	
	Similarly we show that
	$$v\leq\tilde{v}\ \mbox{for all}\ v\in S_-.$$
The proof is now complete.
\end{proof}

We are now ready to produce extremal constant sign solutions for problem (\ref{eq1}), that is, a smallest positive solution $\bar{u}_+$ and a biggest negative solution $\bar{v}_-$. In the next section, using $\bar{u}_+$ and $\bar{v}_-$ we will produce a nodal (sign-changing) solution.
\begin{prop}\label{prop9}
	If hypotheses $H(a),H(\xi),H(\beta),H(f)_1$ hold, then problem (\ref{eq1}) admits a smallest positive solution $\bar{u}_+\in D_+$ and a biggest negative solution $\bar{v}_-\in -D_+$.
\end{prop}
\begin{proof}
	From Papageorgiou, R\u{a}dulescu and Repov\v{s} \cite{21} we know that
	\begin{itemize}
		\item $S_+$ is downward directed (that is, if $u_1,u_2\in S_+$, then we can find $u\in S_+$ such that $u\leq u_1$, $u\leq u_2$).
		\item $S_-$ is upward directed (that is, if $v_1,v_2\in S_-$, then we can find $v\in S_-$ such that $v_1\leq v$, $v_2\leq v$).
	\end{itemize}
	
	Then as in the proof of Proposition \ref{prop6} of Papageorgiou and R\u{a}dulescu \cite{19}, we can find $\{u_n\}_{n\geq 1}\subseteq S_+$ such that
	\begin{eqnarray}\label{eq29}
		&&\inf S_+=\inf\limits_{n\geq 1}u_n,\ \tilde{u}\leq u_n\ \mbox{for all}\ n\in\NN\ (\mbox{see Proposition \ref{prop8}}).
	\end{eqnarray}
	
	Evidently, $\{u_n\}_{n\geq 1}\subseteq W^{1,p}(\Omega)$ is bounded. So, we may assume that
	\begin{equation}\label{eq30}
		u_n\stackrel{w}{\rightarrow}\bar{u}_+\ \mbox{in}\ W^{1,p}(\Omega)\ \mbox{and}\ u_n\rightarrow \bar{u}_+\ \mbox{in}\ L^p(\Omega)\ \mbox{and}\ L^p(\partial\Omega).
	\end{equation}
	
	We have
	\begin{eqnarray}\label{eq31}
		&&\left\langle A(u_n),h\right\rangle+\int_{\Omega}\xi(z)u_n^{p-1}hdz+\int_{\partial\Omega}\beta(z)u_n^{p-1}hd\sigma=\int_{\Omega}f(z,u_n)hdz\\
		&&\mbox{for all}\ h\in W^{1,p}(\Omega),\ n\in\NN\nonumber.
	\end{eqnarray}
	
	In (\ref{eq31}) we choose $h=u_n-\bar{u}_+\in W^{1,p}(\Omega)$, pass to the limit as $n\rightarrow\infty$ and use (\ref{eq30}). Then we have
	\begin{eqnarray}\label{eq32}
		&&\lim\limits_{n\rightarrow\infty}\left\langle A(u_n),u_n-\bar{u}_+\right\rangle=0,\nonumber\\
		&\Rightarrow&u_n\rightarrow\bar{u}_+\ \mbox{in}\ W^{1,p}(\Omega)\ (\mbox{see Proposition \ref{prop4}}).
	\end{eqnarray}
	
	So, if in (\ref{eq31}) we pass to the limit as $n\rightarrow\infty$ and use (\ref{eq32}), then
	\begin{eqnarray*}
		&&\left\langle A(\bar{u}_+),h)\right\rangle+\int_{\Omega}\xi(z)\bar{u}_+^{p-1}hdz+\int_{\partial\Omega}\beta(z)\bar{u}_+^{p-1}hd\sigma=\int_{\Omega}f(z,\bar{u}_+)hdz\\
		&&\mbox{for all}\ h\in W^{1,p}(\Omega),\\
		&\Rightarrow&-{\rm div}\, a(D\bar{u}_+(z))+\xi(z)\bar{u}_+(z)^{p-1}=f(z,\bar{u}_+(z))\ \mbox{for almost all}\ z\in\Omega,\\
		&&\frac{\partial \bar{u}_+}{\partial n_a}+\beta(z)\bar{u}_+^{p-1}=0\ \mbox{on}\ \partial\Omega\\
		&&(\mbox{see Papageorgiou and R\u{a}dulescu \cite{16}}),\\
		&\Rightarrow&\bar{u}_+\in C_+\ (\mbox{as before, by the nonlinear regularity theory}).
	\end{eqnarray*}
	
	From (\ref{eq29}) and (\ref{eq32}), we have
	$$\tilde{u}\leq\bar{u}_+,\ \mbox{hence}\ \bar{u}_+\neq 0.$$
	
	As before, via the nonlinear maximum principle, we have
	\begin{eqnarray*}
		&&\bar{u}_+\in D_+,\\
		&\Rightarrow&\bar{u}_+\in S_+\ \mbox{and}\ \bar{u}_+=\inf S_+.
	\end{eqnarray*}
	
	Similarly we produce
	$$\bar{v}_-\in S_-\ \mbox{and}\ \bar{v}_-=\sup S_-.$$
The proof is now complete.
\end{proof}

\section{Existence of Nodal Solutions}

In this section, using the extremal constant sign solutions $\bar{v}_-\in -D_+$ and $\bar{u}_+\in D_+$, we produce a nodal (sign changing) solution. The idea is to use truncation techniques to focus on the order interval $[\bar{v}_-,\bar{u}_+]$. Using variational tools we obtain a solution $y_0$ in this order interval, which is distinct from $0,\bar{v}_-,\bar{u}_+$. The extremality of $\bar{v}_-$ and $\bar{u}_+$ means that this solution $y_0$ is necessarily nodal.
\begin{prop}\label{prop10}
	If hypotheses $H(a),H(\xi),H(\beta),H(f)_1$ hold, then problem (\ref{eq1}) admits a nodal solutions $y_0\in[\bar{v}_-,\bar{u}_+]\cap C^1(\overline{\Omega})$.
\end{prop}
\begin{proof}
	Let $\bar{v}_-\in-D_+$ and $\bar{u}_+\in D_+$ be the two extremal constant sign solutions produced in Proposition \ref{prop9} and let $\vartheta>||\xi||_{\infty}$. We introduce the following Carath\'eodory function
		\begin{eqnarray}\label{eq33}
			\ell(z,x)=\left\{\begin{array}{ll}
				f(z,\bar{v}_-(z))+\vartheta|\bar{v}_-(z)|^{p-2}\bar{v}_-(z)&\mbox{if}\ x<\bar{v}_-(z)\\
				f(z,x)+\vartheta|x|^{p-2}x&\mbox{if}\ \bar{v}_-(z)\leq x\leq\bar{u}_+(z)\\
				f(z,\bar{u}_+(z))+\vartheta\bar{u}_+(z)^{p-1}&\mbox{if}\ \bar{u}_+(z)<x.
			\end{array}\right.
		\end{eqnarray}
		
		We also consider the positive and negative truncations of $\ell(z,\cdot)$, that is, the Carath\'eodory functions
		$$\ell_{\pm}(z,x)=\ell(z,\pm x^{\pm})\ \mbox{for all}\ (z,x)\in\Omega\times\RR.$$
		
		We set $$L(z,x)=\int^x_0\ell(z,s)ds\quad\mbox{and}\quad L_{\pm}(z,x)=\int^x_0\ell_{\pm}(z,s)ds.$$  Consider the $C^1$-functionals $\tilde{\varphi},\tilde{\varphi}_{\pm}:W^{1,p}(\Omega)\rightarrow\RR$ defined by
		\begin{eqnarray*}
			&&\tilde{\varphi}(u)=\int_{\Omega}G(Du)dz+\frac{1}{p}\int_{\Omega}(\xi(z)+\vartheta)|u|^pdz+\frac{1}{p}\int_{\partial\Omega}\beta(z)|u|^pd\sigma-\int_{\Omega}L(z,u)dz\\
			&&\tilde{\varphi}_{\pm}(u)=\int_{\Omega}G(Du)dz+\frac{1}{p}\int_{\Omega}(\xi(z)+\vartheta)|u|^pdz+\\
			&&\hspace{0.5cm}\frac{1}{p}\int_{\partial\Omega}\beta(z)|u|^pd\sigma-\int_{\Omega}L_{\pm}(z,u)dz\ \mbox{for all}\ u\in W^{1,p}(\Omega).
		\end{eqnarray*}
		
		\begin{claim}\label{cl1}
			$K_{\tilde{\varphi}}\subseteq[\bar{v}_-,\tilde{u}_+]\cap C^1(\overline{\Omega}),\ K_{\tilde{\varphi}_{+}}=\{0,\bar{u}_+\},\ K_{\tilde{\varphi}_-}=\{0,\bar{v}_-\}$.
		\end{claim}
		
		Let $u\in K_{\tilde{\varphi}}$. We have
		\begin{eqnarray}\label{eq34}
			&&\left\langle A(\tilde{u}),h\right\rangle+\int_{\Omega}(\xi(z)+\vartheta)|\tilde{u}|^{p-2}\tilde{u}hdz+\int_{\partial\Omega}\beta(z)|\tilde{u}|^{p-2}\tilde{u}hd\sigma=\int_{\Omega}f(z,\tilde{u})hdz\\
			&&\mbox{for all}\ h\in W^{1,p}(\Omega).\nonumber
		\end{eqnarray}
		
		In (\ref{eq34}) let $h=(u-\bar{u}_+)^+\in\ W^{1,p}(\Omega)$. Then
		\begin{eqnarray*}
			&&\left\langle A(u),(u-\bar{u}_+)^+\right\rangle+\int_{\Omega}(\xi(z)+\vartheta)u^{p-1}(u-\bar{u}_+)^+dz+\int_{\partial\Omega}\beta(z)u^{p-1}(u-\bar{u}_+)^+d\sigma\\
			&&=\int_{\Omega}[f(z,\bar{u}_+)+\vartheta\bar{u}_+^{p-1}](u-\bar{u}_+)^+dz\ (\mbox{see (\ref{eq33})})\\
			&&=\left\langle A(\bar{u}_+),(u-\bar{u}_+)^+\right\rangle+\int_{\Omega}(\xi(z)+\vartheta)\bar{u}_+^{p-1}(u-\bar{u}_+)^+dz+\int_{\partial\Omega}\beta(z)\bar{u}_+^{p-1}(u-\bar{u}_+)^+d\sigma\\
			&&(\mbox{since}\ \bar{u}_+\in S_+),\\
			&\Rightarrow&\left\langle A(u)-A(\bar{u}_+),(u-\bar{u}_+)^+\right\rangle+\int_{\Omega}(\xi(z)+\vartheta)(u^{p-1}-\bar{u}_+^{p-1})(u-\bar{u}_+)^+dz+\\
			&&\hspace{0.5cm}\int_{\partial\Omega}\beta(z)(u^{p-1}-\bar{u}_+^{p-1})(u-\bar{u}_+)^+d\sigma=0,\\
			&\Rightarrow&u\leq\bar{u}_+.
		\end{eqnarray*}
		
		Similarly, if in (\ref{eq34}) we choose $h=(\bar{v}_--u)^+\in W^{1,p}(\Omega)$, then we obtain that
		\begin{eqnarray*}
			&&\bar{v}_-\leq u,\\
			&\Rightarrow&u\in[\bar{v}_-,\bar{u}_+]\cap C^1(\overline{\Omega})\ (\mbox{nonlinear regularity theory}),\\
			&\Rightarrow&K_{\tilde{\varphi}}\subseteq[\bar{v}_-,\bar{u}_+]\cap C^1(\overline{\Omega}).
		\end{eqnarray*}
		
		In a similar fashion we show that
		$$K_{\tilde{\varphi}_+}\subseteq[0,\bar{u}_+]\cap C^1(\overline{\Omega})\ \mbox{and}\ K_{\tilde{\varphi}_-}\subseteq[\bar{v}_-,0]\cap C^1(\overline{\Omega}).$$
		
		The extremality of solutions $\bar{u}_+\in D_+$ and $\bar{v}_-\in-D_+$ implies that
		$$K_{\tilde{\varphi}_+}=\{0,\bar{u}_+\}\ \mbox{and}\ K_{\tilde{\varphi}_-}=\{0,\bar{v}_-\}.$$	
		This proves Claim \ref{cl1}.
		\begin{claim}\label{cl2}
			$\bar{u}_+\in D_+$ and $\bar{v}_-\in -D_+$ are local minimizers of $\tilde{\varphi}$.
		\end{claim}
		
		From (\ref{eq33}) and since $\vartheta>||\xi||_{\infty}$ it is clear that $\tilde{\varphi}_+$ is coercive and sequentially weakly lower semicontinuous. So, we can find $\tilde{u}_+\in W^{1,p}(\Omega)$ such that
		\begin{equation}\label{eq35}
			\tilde{\varphi}_+(\tilde{u}_+)=\inf\{\tilde{\varphi}_+(u):u\in W^{1,p}(\Omega)\}.
		\end{equation}
		
		As before, hypothesis $H(f)_1(ii)$ implies that
		\begin{eqnarray*}
			&&\tilde{\varphi}_+(\tilde{u}_+)<0=\tilde{\varphi}_+(0),\\
			&\Rightarrow&\tilde{u}_+\neq 0\ \mbox{and}\ \tilde{u}_+\in K_{\tilde{\varphi}_+}\ (\mbox{see (\ref{eq35})}),\\
			&\Rightarrow&\tilde{u}_+=\bar{u}_+\ (\mbox{see Claim \ref{cl1}}).
		\end{eqnarray*}
		
		Clearly, $\tilde{\varphi}|_{C_+}=\tilde{\varphi}_+|_{C_+}$ (see (\ref{eq33})). Since $\bar{u}_+\in D_+$, it follows that
		\begin{eqnarray*}
			&&\bar{u}_+\ \mbox{is a local}\ C^1(\overline{\Omega})-\mbox{minimizer of}\ \tilde{\varphi},\\
			&\Rightarrow&\bar{u}_+\ \mbox{is a local}\ W^{1,p}(\Omega)-\mbox{minimizer of}\ \tilde{\varphi}\ (\mbox{see Proposition \ref{prop5}}),
		\end{eqnarray*}
		
		Similarly for $\bar{v}_-\in D_+$, using this time the functional $\tilde{\varphi}_-$.		
		This proves Claim \ref{cl2}.

\smallskip		
		We may assume that $K_{\tilde{\varphi}}$ is finite. Otherwise, on account of Claim \ref{cl1} and due to the extremality of $\bar{u}_+$ and $\bar{v}_-$, we already have an infinity of nodal solutions. In addition, without any loss of generality, we may assume that $\tilde{\varphi}(\bar{v}_-)\leq\tilde{\varphi}(\bar{u}_+)$ (the reasoning is similar if the opposite inequality holds). Claim \ref{cl2} implies that we can find $\rho\in(0,1)$ small such that
		\begin{equation}\label{eq36}
			\tilde{\varphi}(\bar{v}_-)\leq\tilde{\varphi}(\bar{u}_+)<\inf\{\tilde{\varphi}(u):||u-\bar{u}_+||=\rho\}=\tilde{m}_{\rho},\ ||\bar{v}_--\bar{u}_+||>\rho
		\end{equation}
		(see Aizicovici, Papageorgiou and Staicu \cite{1}, proof of Proposition 29).
		
		Evidently, $\tilde{\varphi}$ is coercive (see (\ref{eq33})). Therefore
		\begin{equation}\label{eq37}
			\tilde{\varphi}\ \mbox{satisfies the C-condition}.
		\end{equation}
		
		Then (\ref{eq36}) and (\ref{eq37}) permit the use of the mountain pass theorem (see Theorem \ref{th1}). So, we can find
		\begin{equation}\label{eq38}
			y_0\in K_{\tilde{\varphi}}\ \mbox{and}\ \tilde{m}_{\rho}\leq \tilde{\varphi}(y_0).
		\end{equation}
		
		From (\ref{eq38}), (\ref{eq36}) and Claim \ref{cl1} we have that
		$$y_0\in[\bar{v}_-,\bar{u}_+]\cap C^1(\overline{\Omega})\ \mbox{solves (\ref{eq1}) and}\ y_0\neq\bar{u}_+,\ y_0\neq\bar{v}_-.$$
		
		Since $y_0\in C^1(\overline{\Omega})$ is a critical point of mountain pass type for $\tilde{\varphi}$, we have
		\begin{equation}\label{eq39}
			C_1(\tilde{\varphi},y_0)\neq 0
		\end{equation}
		(see Motreanu, Motreanu and Papageorgiou \cite[Corollary 6.81, p. 168]{13}).
		
		On the other hand, hypothesis $H(f)_1(ii)$ implies that
		\begin{equation}\label{eq40}
			C_k(\tilde{\varphi},0)=0\ \mbox{for all}\ k\in\NN_0
		\end{equation}
		(see Papageorgiou and R\u{a}dulescu \cite[Proposition 6]{17}). Comparing (\ref{eq39}) and (\ref{eq40}) we infer that $y_0\neq 0$. Therefore we conclude that $y_0\in[\bar{v}_-,\bar{u}_+]\cap C^1(\overline{\Omega})$ is a local solution for problem (\ref{eq1}).
\end{proof}

So, we can formulate the following multiplicity theorem for problem (\ref{eq1}).
\begin{theorem}\label{th11}
	If hypotheses $H(a),H(\xi),H(\beta),H(f)_1$ hold, then problem (\ref{eq1}) has at least three nontrivial smooth solutions
	$$u_0\in D_+,v_0\in -D_+,\ \mbox{and}\ y_0\in[v_0,u_0]\cap C^1(\overline{\Omega})\ \mbox{nodal}.$$
\end{theorem}

\section{Semilinear Equations}

In this section, we introduce a special case of problem (\ref{eq1}) in which $a(y)=y$ for all $y\in \RR^N$ (that is, the differential operator is the Laplacian, which corresponds to a semilinear equation). So, the problem under consideration is the following
\begin{equation}\label{eq41}
	\left\{\begin{array}{l}
		-\Delta u(z)+\xi(z)u(z)=f(z,u(z))\ \mbox{in}\ \Omega,\\
		\frac{\partial u}{\partial n}+\beta(z)u=0\ \mbox{on}\ \partial\Omega.
	\end{array}\right\}
\end{equation}

In this case we can also relax the conditions on the potential function $\xi(\cdot)$ and allow it to be unbounded. For problem (\ref{eq41}) we were able to improve Theorem \ref{th11} and produce a second nodal solution for a local of four nontrivial smooth solutions.

Now the hypotheses on the data of (\ref{eq41}) are the following:

\smallskip
$\bullet\ $ $H(\xi)':$ $\xi\in L^s(\Omega)$ with $s>N,\ \xi^+\in L^{\infty}(\Omega)$;

\smallskip
$\bullet\ $ $H(\beta)':$ $\beta\in W^{1,\infty}(\Omega),\ \beta(z)\geq 0$ for all $z\in\partial\Omega$.
\begin{remark}
	Again we can have $\beta\equiv 0$, which corresponds to the Neumann problem.
\end{remark}

$H(f)_2:$ the function $f:\Omega\times\RR\rightarrow\RR$ is  measurable and for almost all $z\in\Omega$, $f(z,0)=0,\ f(z,\cdot)\in C^1(\RR)$ and
\begin{itemize}
	\item[(i)] there exist $\eta>0$ and $a_{\eta}\in L^{\infty}(\Omega)_+$ such that
	$$|f(z,x)|\leq a_{\eta}(z)\ \mbox{for almost all}\ z\in\Omega,\ \mbox{and all}\ x\in[-\eta,\eta];$$
	\item[(ii)] there exist $m\in\NN,\ m\geq 2$ and $\delta_0>0$ such
	$$\hat{\lambda}_mx^2\leq f(z,x)x\leq\hat{\lambda}_{m+1}x^2\ \mbox{for almost all}\ z\in\Omega,\ \mbox{and all}\ |x|\leq\delta_0;$$
	\item[(iii)] with $\eta>0$ as in (i) we have
	$$f(z,\eta)-\xi(z)\eta\leq 0\leq f(z,-\eta)+\xi(z)\eta\ \mbox{for almost all}\ z\in\Omega$$
	and there exists $\hat{\xi}_{\eta}>0$ such that for almost all $z\in\Omega$
	$$x\mapsto f(z,x)+\hat{\xi}_{\eta}x$$
	is nondecreasing on the interval $[-\eta,\eta]$.
\end{itemize}

We have the following multiplicity theorem for problem (\ref{eq41}).
\begin{theorem}\label{th12}
	If hypotheses $H(\xi)',H(\beta)',H(f)_2$ hold, then problem (\ref{eq41}) has at least four nontrivial smooth solutions
	$$u_0\in D_+,\ v_0\in-D_+,\ \mbox{and}\ y_0,\, \hat{y}\in int_{C^1(\overline{\Omega})}[v_0,u_0]\ \mbox{nodal}.$$
\end{theorem}
\begin{proof}
	Since we have relaxed the conditions on the potential function $\xi(\cdot)$ and on the boundary coefficient $\beta(\cdot)$, we need to show how the solutions of problem (\ref{eq1}) exhibit the global (that is, up to the boundary) regularity properties claimed by the theorem.
	
	So, let $u\in[-\eta,\eta]$ be a nontrivial solution of (\ref{eq41}). Then
	\begin{eqnarray}\label{eq42}
		&&\left\langle A(u),h\right\rangle+\int_{\Omega}\xi(z)uhdz+\int_{\partial\Omega}\beta(z)uhd\sigma=\int_{\Omega}f(z,u)hdz\ \mbox{for all}\ h\in H^1(\Omega)\nonumber\\
		&\Rightarrow&-\Delta u(z)+\xi(z)u(z)=f(z,u(z))\ \mbox{for almost all}\ z\in\Omega,\ \frac{\partial u}{\partial n}+\beta(z)u=0\ \mbox{on}\ \partial\Omega\\
		&&(\mbox{see Papageorgiou and R\u{a}dulescu \cite{16}})\nonumber
	\end{eqnarray}
	
	Let
	\begin{eqnarray*}
		g_1(z)=\left\{\begin{array}{ll}
			0&\mbox{if}\ |u(z)|\leq 1\\
			\frac{f(z,u(z))}{u(z)}&\mbox{if}\ |u(z)|>1
		\end{array}\right.\ \mbox{and}\ g_2(z)=\left\{\begin{array}{ll}
			f(z,u(z))&\mbox{if}\ |u(z)|\leq 1\\
			0&\mbox{if}\ |u(z)|>1.
		\end{array}\right.
	\end{eqnarray*}
	
	Note that hypotheses $H(f)_2(i),(ii)$ imply that
	$$|f(z,x)|\leq c_{12}|x|\ \mbox{for almost all}\ z\in\Omega,\ \mbox{all}\ x\in[-\eta,\eta],\ \mbox{some}\ c_{12}>0.$$
	
	Then it follows that $g_1,g_2\in L^{\infty}(\Omega)$. We rewrite (\ref{eq42}) as follows:
	$$\left\{\begin{array}{l}
		-\Delta u(z)=(g_1(z)-\xi(z))u(z)+g_2(z)\ \mbox{for almost all}\ z\in\Omega,\\
		\frac{\partial u}{\partial n}+\beta(z)u=0\ \mbox{on}\ \partial\Omega.
	\end{array}\right\}$$
	
	Note that $g_1-\xi\in L^s(\Omega)\ s>N$ (see hypothesis $H(\xi)'$). Invoking Lemma 5.1 of Wang \cite{27}, we have $u\in L^{\infty}(\Omega)$. Then the Calderon-Zygmund estimates (see Wang \cite[Lemma 5.2]{27}), imply that
	\begin{eqnarray*}
		&&u\in W^{2,s}(\Omega),\\
		&\Rightarrow&u\in C^{1,\alpha}(\overline{\Omega})\ \mbox{with}\ \alpha=1-\frac{N}{s}>0\\
		&&(\mbox{by the Sobolev embedding theorem}).
	\end{eqnarray*}
	
	Now the condition near zero (hypothesis $H(f)_2(ii)$) is different. Here, $f(z,\cdot)$ is linear near zero, while hypothesis $H(f)_1(ii)$ implies the presence of a concave nonlinearity near zero. So, we need to verify that Theorem \ref{th11} remains valid also in the present setting. Note that now, given $r>2$, we can find $\vartheta_0=\vartheta_0(r)>0$ such that
	$$f(z,x)x\geq\hat{\lambda}_mx^2-\vartheta_0|x|^r\ \mbox{for almost all}\ z\in\Omega,\ \mbox{and all}\ |x|\leq\eta.$$
	
	We introduce the following Carath\'eodory function
	\begin{equation}\label{eq43}
		\mu(z,x)=\left\{\begin{array}{ll}
			-\hat{\lambda}_m\eta+\vartheta_0\eta^{r-1}&\mbox{if}\ x<-\eta\\
			\hat{\lambda}_mx-\vartheta_0|x|^{r-2}x&\mbox{if}\ -\eta\leq x\leq\eta\\
			\hat{\lambda}_m\eta-\vartheta_0\eta^{r-1}&\mbox{if}\ \eta<x
		\end{array}\right.
	\end{equation}
	and then consider the following auxiliary Robin problem
	\begin{equation}\label{eq44}
		\left\{\begin{array}{l}
			-\Delta u(z)+\xi(z)u(z)=\mu(z,u(z))\ \mbox{in}\ \Omega,\\
			\frac{\partial u}{\partial n}+\beta(z)u=0\ \mbox{on}\ \partial\Omega
		\end{array}\right\}
	\end{equation}
	(see (\ref{eq10}) and (\ref{eq11}) for the corresponding items in the previous setting). Reasoning as in the proof of Proposition \ref{prop6}, we obtain a unique positive solution
	$$\tilde{u}\in[0,\eta]\cap D_+$$
	for problem (\ref{eq44}). Then $\tilde{v}=-\tilde{u}\in[-\eta,0]\cap(-D_+)$ is the unique negative solution of (\ref{eq44}).
	
	In fact in this case, due to the semilinearity of the problem, we can have an alternative more direct proof of the uniqueness of the positive solution of problem (\ref{eq44}). So, suppose that $\tilde{u},\hat{u}$ are two positive solutions of (\ref{eq44}). We have
	\begin{equation}\label{eq45}
		\tilde{u},\hat{u}\in[0,\eta]\cap D_+.
	\end{equation}
	
	Let $t^*>0$ be the biggest real number such that
	\begin{equation}\label{eq46}
		t^*\hat{u}\leq\tilde{u}.
	\end{equation}
	
	Assume that $0<t^*<1$. Evidently, we can find $\tilde{\xi}_{\eta}>0$ such that the function
	$$x\mapsto (\hat{\lambda}_m+\tilde{\xi}_{\eta})x-\vartheta_0|x|^{r-2}x$$
	is nondecreasing on $[-\eta,\eta]$. We have
	\begin{eqnarray*}
		&&-\Delta\tilde{u}(z)+(\xi(z)+\tilde{\xi}_{\eta})\tilde{u}(z)\\
		&=&\mu(z,\tilde{u}(z))+\tilde{\xi}_{\eta}\tilde{u}(z)\\
		&=&(\hat{\lambda}_m+\tilde{\xi}_{\eta})\tilde{u}(z)-\vartheta_0\tilde{u}(z)^{r-1}\ (\mbox{see (\ref{eq45}) and (\ref{eq43})})\\
		&\geq&(\hat{\lambda}_m+\tilde{\xi}_{\eta})(t^*\hat{u}(z))-\vartheta_0(t^*\hat{u}(z))^{r-1}\ (\mbox{see (\ref{eq46})})\\
		&\geq&t^*[\hat{\lambda}_m\hat{u}(z)-\vartheta_0\hat{u}(z)^{r-1}+\tilde{\xi}_{\eta}\hat{u}(z)]\ (\mbox{since}\ 0<t^*<1,r>2)\\
		&=&-\Delta(t^*\hat{u}(z))+(\xi(z)+\tilde{\xi}_{\eta})(t^*\hat{u}(z))\ \mbox{for almost all}\ z\in\Omega,\\
		\Rightarrow&&\Delta(\tilde{u}-t^*\hat{u})(z)\leq[||\xi^+||_{\infty}+\tilde{\xi}_{\eta}](\tilde{u}-t^*\hat{u})(z)\ \mbox{for almost all}\ z\in\Omega\\
		&&(\mbox{see hypothesis}\ H(\xi))\\
		\Rightarrow&&\tilde{u}-t^*\hat{u}\in D_+\ (\mbox{by the strong maximum principle}).
	\end{eqnarray*}
	
	This contradicts the maximality of $t^*$. Hence $t^*\geq 1$ and so
	$$\hat{u}\leq\tilde{u}\ (\mbox{see (\ref{eq46})}).$$
	
	Interchanging the roles of $\tilde{u}$ and $\hat{u}$ in the above argument, we also have
	\begin{eqnarray*}
		&&\tilde{u}\leq\hat{u},\\
		&\Rightarrow&\tilde{u}=\hat{u}.
	\end{eqnarray*}
	
	This is an alternative, more direct proof of the uniqueness of positive and negative solutions of problem (\ref{eq44}).
	
	As in the proof of Proposition \ref{prop7}, we introduce the Carath\'eodory function
	$$\hat{f}_+(z,x)=\left\{\begin{array}{ll}
		0&\mbox{if}\ x<0\\
		f(z,x)+\mu_0x&\mbox{if}\ 0\leq x\leq \eta\\
		f(z,\eta)+\mu_0\eta&\mbox{if}\ \eta<x
	\end{array}\right.$$
	with $\mu_0>||\xi^+||_{\infty}$ (see hypothesis $H(\xi)'$).
	
	We set $\hat{F}_+(z,x)=\int^x_0\hat{f}_+(z,s)ds$ and consider the $C^1$-functional $\tilde{\varphi}_+:H^1(\Omega)\rightarrow\RR$ defined by
	$$\hat{\varphi}_+(u)=\frac{1}{2}\tau_0(u)+\frac{\mu_0}{2}||u||^2_2-\int_{\Omega}\hat{F}_+(z,u)dz\ \mbox{for all}\ u\in H^1(\Omega),$$
	with $\tau_0(u)=||Du||^2_2+\int_{\Omega}\xi(z)u^2dz+\int_{\partial\Omega}\beta(z)u^2d\sigma$ for all $u\in H^1(\Omega)$. Using the direct method of the calculus of variations, we obtain $u_0\in H^1(\Omega)$ such that
	$$\hat{\varphi}_+(u_0)=\inf\{\hat{\varphi}_+(u):u\in H^1(\Omega)\}.$$
	
	Since $m\geq 2$ (see hypothesis $H(f)_2(ii)$), we have
	\begin{eqnarray*}
		&&\hat{\varphi}_+(u_0)<0=\hat{\varphi}_+(0),\\
		&\Rightarrow&u_0\neq 0.
	\end{eqnarray*}
	
	As in the proof of Proposition \ref{prop7}, we show that
	$$u_0\in[0,\eta]\cap D_+.$$
	
	Similarly, using the Carath\'eodory function
	$$\hat{f}_-(z,x)=\left\{\begin{array}{ll}
		f(z,-\eta)-\mu_0\eta&\mbox{if}\ x<-\eta\\
		f(z,x)+\mu_0x&\mbox{if}\ -\eta\leq x\leq 0\\
		0&\mbox{if}\ 0<x,
	\end{array}\right.$$
	we produce a negative solution
	$$v_0\in[-\eta,0]\cap(-D_+).$$
	
	Therefore we have
	$$\emptyset\neq S_+\subseteq D_+\ \mbox{and}\ \emptyset\neq S_-\subseteq-D_+.$$
	
	In addition, as in the proof of Proposition \ref{prop8}, we show that
	$$\tilde{u}\leq u\ \mbox{for all}\ u\in S_+\ \mbox{and}\ v\leq\tilde{v}\ \mbox{for all}\ v\in S_-.$$
	
	Moreover, reasoning as in the proof of Proposition \ref{prop9}, we produce extremal constant sign solution for problem (\ref{eq41})
	$$\bar{u}_+\in D_+\ \mbox{and}\ \bar{v}_-\in-D_+.$$
	
As in the proof of Proposition \ref{prop10},	using these two extremal constant sign solutions,  we introduce the functional $\tilde{\varphi}$ and using it we produce a nodal solution. Note that in (\ref{eq33}) we replace $\vartheta>0$ by $\mu_0>0$ and we set $p=2$. Claims \ref{cl1} and \ref{cl2} in the proof of Proposition \ref{prop10} remain valid (as before, since $m\geq 2$, we have $\tilde{\varphi}_+(\tilde{u}_+)<0=\tilde{\varphi}_+(0)$ and so $\tilde{u}_+\neq 0$). Finally, we apply the mountain pass theorem (see Theorem \ref{th1}) and obtain
	\begin{equation}\label{eq47}
		y_0\in K_{\tilde{\varphi}}\subseteq[\bar{v}_-,\bar{u}_+]\cap C^1(\overline{\Omega}),\ y_0\notin\{\bar{v}_-,\bar{u}_+\}.
	\end{equation}
	
	Therefore we have
	\begin{equation}\label{eq48}
		C_1(\tilde{\varphi},y_0)\neq 0
	\end{equation}
	(see Motreanu, Motreanu and Papageorgiou \cite[Corollary 6.81, p. 168]{13}). In this case, since the condition near zero is different (see $H(f)_1(ii)$ and compare it with $H(f)_1(ii)$), relation (\ref{eq40}) is no longer true. We need to compute the critical groups of $\tilde{\varphi}$ at $u=0$.
	\begin{claim}\label{claim3}
	We have	$C_k(\tilde{\varphi},0)=\delta_{k,d_m}\ZZ$ for all $k\in\NN_0$, with $d_m=dim\overset{m}{\underset{\mathrm{k=1}}\oplus}E(\hat{\lambda}_k)\geq 2$.
	\end{claim}
	
	Let $\lambda\in(\hat{\lambda}_m,\hat{\lambda}_{m+1})$ and consider the $C^2$-functional $\psi:H^1(\Omega)\rightarrow\RR$ defined by
	$$\psi(u)=\frac{1}{2}\tau_0(u)-\frac{\lambda}{2}||u||^2_2\ \mbox{for all}\ u\in H^1(\Omega).$$
	
	We consider the homotopy
	$$h(t,u)=(1-t)\tilde{\varphi}(u)+t\psi(u)\ \mbox{for all}\ t\in[0,1],\ \mbox{all}\ u\in H^1(\Omega).$$
	
	Suppose that we can find $\{t_n\}_{n\geq 1}\subseteq[0,1]$ and $\{u_n\}_{n\geq 1}\subseteq H^1(\Omega)$ such that
	\begin{equation}\label{eq49}
		t_n\mapsto t,u_n\rightarrow 0\ \mbox{in}\ H^1(\Omega)\ \mbox{and}\ h'_u(t_n,u_n)=0\ \mbox{for all}\ n\in\NN.
	\end{equation}
	
	From the equation in (\ref{eq49}), we have
	\begin{equation}\label{eq50}
		\left\{\begin{array}{l}
			-\Delta u_n(z)+\xi(z)u_n(z)-(1-t_n)u^-_n(z)=(1-t_n)f(z,u_n(z))+\lambda t_nu_n(z)\ \mbox{in}\ \Omega,\\
			\frac{\partial u_n}{\partial n}+\beta(z)u_n=0\ \mbox{on}\ \partial\Omega\,.
		\end{array}\right\}
	\end{equation}
	
	From (\ref{eq50}) and the regularity theory of Wang \cite{27}, we know that we can find $\alpha\in(0,1)$ and $c_{13}>0$ such that
	\begin{equation}\label{eq51}
		u_n\in C^{1,\alpha}(\overline{\Omega})\ \mbox{and}\ ||u_n||_{C^{1,\alpha}(\overline{\Omega})}\leq c_{13}\ \mbox{for all}\ n\in\NN.
	\end{equation}
	
	Exploiting the compact embedding of $C^{1,\alpha}(\overline{\Omega})$ into $C^1(\overline{\Omega})$,  we infer from (\ref{eq51}) and (\ref{eq49}) that
	\begin{eqnarray*}
		&&u_n\rightarrow 0\ \mbox{in}\ C^1(\overline{\Omega}),\\
		&\Rightarrow&u_n\in[\bar{v}_-,\bar{u}_+]\cap C^1(\overline{\Omega})\ \mbox{for all}\ n\geq n_0,\\
		&\Rightarrow&\{u_n\}_{n\geq n_0}\subseteq K_{\tilde{\varphi}}\ (\mbox{see Claim \ref{cl1} in the proof of Proposition \ref{prop10}}).
	\end{eqnarray*}
	
	But we have assumed that $K_{\tilde{\varphi}}$ is finite (otherwise, on account of the definition of $\tilde{\varphi}$ and (\ref{eq33}) with $\mu_0>0$ replacing $\vartheta>0$ and $p=2$ and using Claim \ref{cl1} in the proof of Proposition \ref{prop10}, we see that we have an infinity of nodal solutions of (\ref{eq41}) and so we are done). Therefore we have a contradiction and this means that (\ref{eq49}) cannot happen. Then using the homotopy invariance property of critical groups (see Gasinski and Papageorgiou \cite[Theorem 5.125, p. 836]{7}), we have
	\begin{equation}\label{eq52}
		C_k(\tilde{\varphi},0)=C_k(\psi,0)\ \mbox{for all}\ k\in\NN_0.
	\end{equation}
	
	Note that $\psi\in C^2(H^1(\Omega))$. Since $\lambda\in(\hat{\lambda}_m,\hat{\lambda}_{m+1}),u=0$ is a nondegenerate critical point of $\psi$. So, from Gasinski and Papageorgiou \cite[Theorem 5.106, p. 832]{7}, we have
	\begin{eqnarray*}	
		&&C_k(\psi,0)=\delta_{k,d_m}\ZZ\ \mbox{for all}\ k\in\NN_0,\\
		&\Rightarrow&C_k(\tilde{\varphi},0)=\delta_{k,d_m}\ZZ\ \mbox{for all}\ k\in\NN_0\ (\mbox{see (\ref{eq52})}).
	\end{eqnarray*}
	
	This proves Claim \ref{claim3}.
	
	From (\ref{eq48}), (\ref{eq47}) and Claim \ref{claim3}, we infer that
	\begin{eqnarray*}
		&&y_0\notin\{0,\bar{u}_+,\bar{v}_-\}\\
		&\Rightarrow&y_0\ \mbox{is a nodal solution of (\ref{eq41}) and $y_0\in[\bar{v}_-,\bar{u}_+]\cap C^1(\overline{\Omega})$}.
	\end{eqnarray*}
	
	Let $\hat{\xi}_{\eta}>0$ be as postulated by hypothesis $H(f)_2(iii)$. Then
	\begin{eqnarray*}
		&&-\Delta y_0(z)+(\xi(z)+\hat{\xi}_{\eta})y_0(z)\\
		&&=f(z,y_0(z))+\hat{\xi}_{\eta}y_0(z)\\
		&&\leq f(z,\bar{u}_+(z))+\hat{\xi}_{\eta}\bar{u}_+(z)\ (\mbox{see (\ref{eq47}) and hypothesis}\ H(f)_2(iii))\\
		&&=-\Delta\bar{u}_+(z)+(\xi(z)+\hat{\xi}_{\eta})\bar{u}_+(z)\ \mbox{for almost all}\ z\in\Omega,\\
		&\Rightarrow&\Delta(\bar{u}_+-y_0)(z)\leq[||\xi^+||_{\infty}+\hat{\xi}_{\eta}](\bar{u}_+-y_0)(z)\ \mbox{for almost all}\ z\in\Omega\\
		&&(\mbox{see hypothesis}\ H(\xi)')\\
		&\Rightarrow&\bar{u}_+-y_0\in D_+\ (\mbox{by the strong maximum principle}).
	\end{eqnarray*}
	
	Similarly we show that
	$$y_0-\bar{v}_-\in D_+.$$
	
	Therefore we conclude that
	\begin{equation}\label{eq53}
		y_0\in {\rm int}_{C^1(\overline{\Omega})}[\bar{v}_-,\bar{u}_+].
	\end{equation}
	
	From (\ref{eq48}), (\ref{eq53}) and Proposition 5.124 of Gasinski and Papageorgiou \cite[p. 836]{7}, we have
	\begin{equation}\label{eq54}
		C_k(\tilde{\varphi},y_0)=\delta_{k,1}\ZZ\ \mbox{for all}\ k\in\NN_0.
	\end{equation}
	
	Recall that $\bar{u}_+\in D_+$ and $\bar{v}_-\in-D_+$ are local minimizers of $\tilde{\varphi}$ (see Claim \ref{cl2}) in the proof of Proposition \ref{prop10}. Therefore we have
	\begin{equation}\label{eq55}
		C_k(\tilde{\varphi},\bar{u}_+)=C_k(\tilde{\varphi},\bar{v}_-)=\delta_{k,0}\ZZ\ \mbox{for all}\ k\in\NN_0.
	\end{equation}
	
	From the proof of Claim \ref{claim3}, we have
	\begin{equation}\label{eq56}
		C_k(\tilde{\varphi},0)=\delta_{k,d_m}\ZZ\ \mbox{for all}\ k\in\NN_0.
	\end{equation}
	
	Finally, since $\tilde{\varphi}$ is coercive, we have
	\begin{equation}\label{eq57}
		C_k(\tilde{\varphi},\infty)=\delta_{k,0}\ZZ\ \mbox{for all}\ k\in\NN_0.
	\end{equation}
	
	Suppose that $K_{\tilde{\varphi}}=\{0,\bar{u}_+,\bar{v}_-,y_0\}$. Then from (\ref{eq54}), (\ref{eq55}), (\ref{eq56}), (\ref{eq57}), and the Morse relation with $t=-1$ (see (\ref{eq8})), we have
	\begin{eqnarray*}
		&&(-1)^{d_m}+2(-1)^0+(-1)^1=(-1)^0,\\
		&\Rightarrow&(-1)^{d_m}=0,\ \mbox{a contradiction}.
	\end{eqnarray*}
	
	So, there exists $\hat{y}\in H^1(\Omega)$ such that
	\begin{eqnarray*}
		&&\hat{y}\in K_{\tilde{\varphi}},\ \hat{y}\notin\{0,\bar{u}_+,\bar{v}_-,y_0\},\\
		&\Rightarrow&\hat{y}\in[\bar{v}_-,\ \bar{u}_+]\cap C^1(\overline{\Omega}),\hat{y}\notin\{0,\bar{u}_+,\bar{v}_-,y_0\},\\
		&\Rightarrow&\hat{y}\ \mbox{is a second solution of (\ref{eq41}) distinct from}\ y_0.
	\end{eqnarray*}
	
	Moreover, as we did for $y_0$, using hypothesis $H(f)_2(iii)$, we show that
	$$\hat{y}\in {\rm int}_{C^1(\overline{\Omega})}[\bar{v}_-,\bar{u}_+].$$
The proof is now complete.
\end{proof}

\section{Infinitely Many Nodal Solutions}

In this section we return to problem (\ref{eq1}) and by introducing a symmetry condition on $f(z,\cdot)$, we produce a whole sequence of distinct nodal solutions for problem (\ref{eq1}).

The hypotheses on the reaction term are the following:

\smallskip
$H(f)_3:$ $f:\Omega\times\RR\rightarrow\RR$ is a Carath\'eodory function such that for almost all $z\in\Omega$, $f(z,0)=0$ and
\begin{itemize}
	\item[(i)] there exist $\eta>0$ and $a_{\eta}\in L^{\infty}(\Omega)_+$ such that for almost all $z\in\Omega$, $f(z,\cdot)|_{[-\eta,\eta]}$ is odd and
	\begin{equation*}
		|f(z,x)|\leq a_{\eta}(z)\ \mbox{for almost all}\ z\in\Omega,\ \mbox{and all}\ |x|\leq\eta;
	\end{equation*}
	\item[(ii)] with $\tau\in(1,p)$ as in hypothesis $H(a)(iv)$ we have
	$$\lim\limits_{x\rightarrow 0}\frac{f(z,x)}{|x|^{\tau-2}x}=+\infty\ \mbox{uniformly for almost all}\ z\in\Omega.$$
\end{itemize}
\begin{remark}
	The symmetric condition on $f(z,\cdot)$, permits the relaxation of the condition near zero (compare $H(f)_3(ii)$ with $H(f)_1(ii)$). We have also  dropped hypothesis $H(f)_3(iii)$.
\end{remark}

Fix $\lambda(\cdot)$ an even continuous function such that
\begin{itemize}
	\item $\lambda|_{[-c,c]}\equiv 1$ for some $c\in(0,\eta)$;
	\item ${\rm supp}\,\lambda\subseteq(-\eta,\eta)$;
	\item $0\leq\lambda\leq 1$.
\end{itemize}

Let $\hat{f}(z,x)=\lambda(x)f(z,x)+(1-\lambda(x))\xi(z)|x|^{p-2}x$. Evidently, $\hat{f}(z,x)$ is a Carath\'eodory function with the following properties:
\begin{itemize}
	\item for all $z\in\Omega,\ \hat{f}(z,\cdot)$ is odd;
	\item $\hat{f}(z,x)=f(z,x)$ for all $z\in\Omega$,  $|x|\leq c$;
	\item $\hat{f}(z,x)=\xi(z)|x|^{p-2}x$ for all $z\in\Omega$,  $|x|\geq\eta$.
\end{itemize}

It follows that $\hat{f}(z,\eta)-\xi(z)\eta^{p-1}=0=\hat{f}(z,-\eta)+\xi(z)\eta^{p-1}$ (that is, $\hat{f}(z,x)$ satisfies hypothesis $H(f)_1(iii)$).

We consider problem (\ref{eq1}) with $f$ replaced by $\hat{f}$.

Note that given any $\hat{\eta}>0$ and $r>p$, we can find $c_{14}=c_{14}(\hat{\eta},r)>0$ such that
$$\hat{f}(z,x)x\geq\hat{\eta}|x|^{\tau}-c_{14}|x|^r\ \mbox{for almost all}\ z\in\Omega,\ \mbox{and all}\ x\in\RR.$$

Then we introduce the following Carath\'eodory function
$$\mu(z,x)=\left\{\begin{array}{ll}
	-\hat{\eta}\eta^{\tau-1}+c_{14}\eta^{r-1}&\mbox{if}\ x<-\eta\\
	\hat{\eta}|x|^{\tau-2}x-c_{14}|x|^{r-2}x&\mbox{if}\ -\eta\leq x\leq\eta\\
	\hat{\eta}\eta^{\tau-1}-c_{14}\eta^{r-1}&\mbox{if}\ \eta<x.
\end{array}\right.$$

Using this $\mu(\cdot,\cdot)$, we consider the auxiliary Robin problem (\ref{eq11}). As in Proposition \ref{prop6} we show that the auxiliary problem has a unique positive solution $\tilde{u}\in[0,\eta]\cap D_+$ and due to the oddness of the equation, $\tilde{v}=-\tilde{u}\in[-\eta,0]\cap(-D_+)$ is the unique negative solution of the auxiliary problem.

Recalling that we consider problem (\ref{eq1}) with $f(z,x)$ replaced by $\hat{f}(z,x)$, as before we introduce the following sets:
\begin{center}
$S_+$\ =\ the set of positive solutions of (\ref{eq1}) in $[0,\eta]$,\\
$S_-$\ =\ the set of negative solutions of (\ref{eq1}) in $[-\eta,0]$.
\end{center}

If $S_+\neq\emptyset$ and $S_-\neq\emptyset$, then we have
$$\tilde{u}\leq u\ \mbox{for all}\ u\in S_+\ \mbox{and}\ v\leq\tilde{v}\ \mbox{for all}\ v\in S_-\ (\mbox{see Proposition \ref{prop8}}).$$

This leads to the existence of extremal constant sign solutions
$$\bar{u}_+\in D_+\ \mbox{and}\ \bar{v}_-\in-D_+.$$

Using these extremal constant sign solutions, we consider the Carath\'eodory function $\ell(z,x)$ as in (\ref{eq33}) with $f(z,x)$ replaced by $\hat{f}(z,x)$ (see the proof of Proposition \ref{prop10}) and then introduce the $C^1$-functional $\tilde{\varphi}:W^{1,p}(\Omega)\rightarrow\RR$ defined by
\begin{eqnarray*}
	&&\tilde{\varphi}(u)=\int_{\Omega}G(Du)dz+\frac{1}{p}\int_{\Omega}(\xi(z)+\vartheta)|u|^pdz+\frac{1}{p}\int_{\partial\Omega}\beta(z)|u|^pd\sigma-\int_{\Omega}L(z,u)dz\\
	&&\mbox{for all}\ u\in W^{1,p}(\Omega).
\end{eqnarray*}

Here, as before, $\vartheta>||\xi||_{\infty}$ and $L(z,x)=\int^x_0\ell(z,s)ds$.

We know that
\begin{equation}\label{eq58}
	K_{\tilde{\varphi}}\subseteq[v_-,\bar{u}_+]\cap C^1(\overline{\Omega})
\end{equation}
(see Claim \ref{cl1} in the proof of Proposition \ref{prop10}).
\begin{prop}\label{prop13}
	If hypotheses $H(a),H(\xi),H(\beta),H(f)_3$ hold, $n\in\NN$ and $Y_n\subseteq W^{1,p}(\Omega)$ is an $n$-dimensional subspace, then we can find $\rho_n>0$ such that
	$$\sup\{\tilde{\varphi}(u):u\in Y_n,||u||=\rho_n\}<0.$$
\end{prop}
\begin{proof}
	Hypothesis $H(a)(iv)$ implies that we can find $\rho_1>0$ and $c_{13}>0$ such that
	\begin{eqnarray}\label{eq59}
		&&G_0(t)\leq c_{15}t^{\tau}\ \mbox{for all}\ t\in[0,\rho_1],\nonumber\\
		&\Rightarrow&G(y)\leq c_{15}|y|^{\tau}\ \mbox{for all}\ |y|\leq\rho_1.
	\end{eqnarray}
	
	Also, hypothesis $H(f)_3(ii)$ implies that given $\hat{\eta}>0$ we can find $0<\rho_2\leq\rho_1$ such that
	\begin{eqnarray}\label{eq60}
		&&\hat{F}(z,x)\geq\hat{\eta}|x|^{\tau}\ \mbox{for almost all}\ z\in\Omega,\ \mbox{and all}\ |x|\leq\rho_2\\
		&&(\mbox{here}\ \hat{F}(z,x)=\int^x_0\hat{f}(z,s)ds).\nonumber
	\end{eqnarray}
	
	Let $\rho_3=\min\{\min_{\overline{\Omega}}\bar{u}_+,\min\limits_{\overline{\Omega}}(-\bar{v}_-)\}>0$ (recall that $\bar{u}_+\in D_+,\bar{v}_-\in-D_+$). We can always assume that $0<\rho_2\leq\rho_3$. Since $Y_n$ is finite dimensional, all norms are equivalent and so we can find $\rho_n>0$ such that
	\begin{equation}\label{eq61}
		u\in Y_n,\ ||u||\leq\rho_n\Rightarrow|u(z)|\leq\rho_n\ \mbox{for almost all}\ z\in\Omega.
	\end{equation}
	
	Then from (\ref{eq59}), (\ref{eq60}), (\ref{eq61}) we have
	\begin{eqnarray*}
		&&\tilde{\varphi}(u)\leq c_{15}||Du||^{\tau}_{\tau}-\hat{\eta}||u||^{\tau}_{\tau}\leq[c_{16}-\hat{\eta}c_{17}]||u||^{\tau}\\
		&&\mbox{for some}\ c_{16},c_{17}>0\ \mbox{see (\ref{eq33})}.
	\end{eqnarray*}
	
	Since $\hat{\eta}>0$ is arbitrary, we choose $\hat{\eta}>\frac{c_{16}}{c_{17}}$ and have that
	$$\sup\{\tilde{\varphi}(u):u\in Y_n,||u||=\rho_n\}<0.$$
This completes the proof.
\end{proof}

Now we are ready for the multiplicity result producing a whole sequence of distinct nodal solutions.
\begin{theorem}\label{th14}
	If hypotheses $H(a),H(\xi),H(\beta),H(f)_3$ hold, then problem (\ref{eq1}) admits a sequence $\{u_n\}_{n\geq 1}\subseteq C^1(\overline{\Omega})$ of distinct nodal solutions such that $u_n\rightarrow 0$ in $C^1(\overline{\Omega})$ as $n\rightarrow\infty$.
\end{theorem}
\begin{proof}
	We know that $\tilde{\varphi}$ is coercive (see (\ref{eq33})). So, $\tilde{\varphi}$ is bounded below and satisfies the C-condition. Also, $\tilde{\varphi}$ is even. These facts, together with Proposition \ref{prop13}, permit the use of Theorem 1 of Kajikiya \cite{10}. Hence we can find a sequence $\{u_n\}_{n\geq 1}\subseteq W^{1,p}(\Omega)$ such that
	\begin{equation}\label{eq62}
		\{u_n\}_{n\geq 1}\subseteq K_{\tilde{\varphi}}\subseteq[\bar{v}_-,\bar{u}_+]\cap C^1(\overline{\Omega})\ \mbox{and}\ u_n\rightarrow 0\ \mbox{in}\ W^{1,p}(\Omega)\ (\mbox{see (\ref{eq58})}).
	\end{equation}
	
	From (\ref{eq62}) and (\ref{eq33}) we see that $\{u_n\}_{n\geq 1}$ are nodal solutions of (\ref{eq1}). Moreover, the nonlinear regularity theory (see Lieberman \cite{11}) and the compact embedding of $C^{1,\alpha}(\overline{\Omega})$ $(0<\alpha<1)$ into $C^1(\overline{\Omega})$ imply that $u_n\rightarrow 0$ in $C^1(\overline{\Omega})$ as $n\rightarrow\infty$. Since $\hat{f}(z,\cdot)$ and $f(z,\cdot)$ coincide near zero, we have thus produced a sequence of distinct nodal solutions for problem (\ref{eq1}).
\end{proof}

\begin{remark}
	Recently, Papageorgiou and R\u{a}dulescu \cite{18} have proved an analogous result for problems with no potential term (that is, $\xi\equiv 0$) and with a reaction term with zeros. Theorem \ref{th14} generalizes the result of Papageorgiou and R\u{a}dulescu \cite{18}. It also extends Theorem 2.10 of Wang \cite{28} where the equation is driven by the $p$-Laplacian with no potential term (that is, $\xi\equiv 0$). Wang produced a sequence of nontrivial solutions $\{u_n\}_{n\geq 1}$, not necessarily nodal, such that $||u_n||_{\infty}\rightarrow 0$.
\end{remark}

\medskip
{\bf Acknowledgments.} This research was supported by the Slovenian Research Agency grants
P1-0292, J1-8131, J1-7025, N1-0064, and N1-0083. V.D.~R\u adulescu acknowledges the support through the Project
MTM2017-85449-P of the DGISPI (Spain). We thank the referee for useful comments and suggestions.

\end{document}